\documentclass[reqno]{
amsart}

\usepackage{hyperref}

 \hypersetup{colorlinks=true, urlcolor=green, citecolor=red, linkcolor=blue}

\usepackage{latexsym,amssymb, amsthm,amsfonts,amsmath,amssymb,amstext,amsthm}
\usepackage[mathscr]{euscript}
\usepackage[abbrev]{amsrefs}
\usepackage[all]{xy}
\usepackage{graphicx}
\usepackage{tikz}
\usepackage{pdfpages}
\usepackage{xcolor}
\usepackage{tabularray}

\newcommand{\cb}{\overline{C}} % normalization of C
\newcommand{\ccc}{\mathscr{C}}% conductor
\newcommand{\cliff}{{\rm Cliff}} % Clifford Index
\newcommand{\cliffd}{{\rm Cld}} % Clifford Dimension
\newcommand{\ff}{{\rm F}}
\newcommand{\fff}{\mathscr{F}} % sheaf F
\newcommand{\G}{\mathscr{G}}
\newcommand{\gap}{{\rm G}}
\newcommand{\gon}{{\rm gon}} % gonality
\newcommand{\kk}{{{\rm K}}}

\newcommand{\mmp}{\mathfrak{m}_{P}} % maximal ideal at P
\newcommand{\mmr}{\mathfrak{m}_{R}} % maximal ideal at R
\newcommand{\nn}{\mathbb{N}}
\newcommand{\obp}{\overline{\mathcal{O}}_P}

\newcommand{\oo}{\mathcal{O}} % structural sheaf
\newcommand{\op}{\mathcal{O}_P} % local ring at P
\newcommand{\pb}{\overline{P}} % point over P
\newcommand{\sss}{{\rm S}}

\newcommand{\sM}{\mathscr{M}}
\newcommand{\ww}{\omega} % dualizing shaef
\newcommand{\wwp}{\omega_{P}} % dualizing shaef

\def\sbsno{(\arabic{section}.\arabic{subsection})\enspace}
\def\defspec#1{\def\headspec{#1}%
  \ifx\headspec\empty % do nothing if #3 is empty
  \else{\unkern\enspace(#1)}\fi
}       % Thm head spec (#3 = bracketed parameter)
\newtheoremstyle{italics}
  {6pt}%      Space above, empty = `usual value'
  {6pt}%      Space below
  {\itshape}% Body font
  {}%         Indent amount (empty = no indent, \parindent = para indent)
  {\bfseries}% Thm head font
  {.}%        Punctuation after thm head
  {.5em}%     Space after thm head: " " = normal interword space;
  {\sbsno \thmname{#1}{\rm \defspec{#3}}}% Thm head spec

\newtheoremstyle{roman}
  {6pt}%      Space above, empty = `usual value'
  {6pt}%      Space below
  {\rmfamily}% Body font
  {}%         Indent amount (empty = no indent, \parindent = para indent)
  {\bfseries}% Thm head font
  {.}%        Punctuation after thm head
  {.5em}%     Space after thm head: " " = normal interword space;
  {\sbsno \thmname{#1}\defspec{#3}}

\theoremstyle{italics}
 \newtheorem{lem}[subsection]{Lemma}

 \newtheorem{prop}[subsection]{Proposition}
 \newtheorem{thm}[subsection]{Theorem}

\theoremstyle{roman}
 \newtheorem{defi}[subsection]{Definition}
 \newtheorem{exam}[subsection]{Example}

 \newtheorem{sbs}[subsection]{} %% \begin{subsection} is predefined

 \newtheorem{clifford}[subsection]{Clifford's Theorem for Integral Curves}

\numberwithin{equation}{subsection}

\numberwithin{equation}{subsection}

\begin{document}

\title{On Clifford dimension for singular curves}

\thanks{}

%\subtitle{with applications to Clifford index, gonality and Koszul cohomology}

%\titlerunning{Max Noether's Theorem, Clifford Index and Koszul Cohomology}        % if too long for running head

\author[L. Feital]{Lia Feital}
\address{Departamento de Matem\'atica, CCE, UFV
Av. P H Rolfs s/n, 36570-000 Vi\c{c}osa MG, Brazil}
\email{liafeital@ufv.br}

\author[N. Galdino]{Naam\~a Galdino}
\address{IMECC, Unicamp,
Rua S\'ergio Buarque de Holanda, 651, 13083-859 Campinas SP, Brazil}
\email{naama@ime.unicamp.br}

\author[R. V. Martins]{Renato Vidal Martins} 
\address{Departamento de Matem\'atica, ICEx, UFMG
Av. Ant\^onio Carlos 6627,
30123-970 Belo Horizonte MG, Brazil.}
\email{renato@mat.ufmg.br}

\author[A. Souza]{\'Atila Souza}
\address{Departamento de Matem\'atica, ICEx, UFMG
Av. Ant\^onio Carlos 6627,
30123-970 Belo Horizonte MG, Brazil.}
\email{desouzaatilafelipe@gmail.com}



%\dedicatory{}


\begin{abstract}
We study the Clifford dimension of an integral curve. To do so, we extend the notion of Clifford index, allowing torsion-free sheaves on its computation. We derive results for arbitrary curves, and then focus on the monomial case. In this context, we obtain combinatorial formulae for the  Clifford index and apply them to the case of Clifford dimension $2$. 
\end{abstract}

\maketitle

% On Clifford dimension for singular curves
% On Enriques-Babbage Theorem for unicuspidal rational curves 

\section{Introduction}\label{Intro}

Let $C$ be an integral and complete curve over an algebraically closed field of arbitrary characteristic. First, assume $C$ is smooth. 
%As it is well known, the Clifford index of a line bundle $\mathcal{L}$ on $C$ measures how far it is from satisfying Clifford's equality. As the latter holds for trivial 
Let $\mathscr{L}$ be a line bundle on $C$ of degree $d$ and $r+1$ independent global sections. Its \emph{Clifford index} is $\cliff(\mathscr{L})=d-2r$. So it measures how far $\mathscr{L}$ is from satisfying Clifford's equality. Set $\cliff(C)=\text{min}\{\cliff(\mathscr{L})\,|\, h^0(\mathscr{L}), h^1(\mathscr{L})\geq 2\}$. As well known \cite{H}*{Thm. IV.5.4}, unless $C$ is hyperelliptic, the trivial and canonical bundles are the only satisfying Clifford's equality. So the Clifford index of a curve measures how far it is from being hyperelliptic. 

This invariant has been studied by numerous authors in various ways, especially in the context of Green's conjecture \cite{Gr}*{Conj.\ (5.1)} (see, for instance, \cite{ApF}). Here, we are particularly interested in a closely related concept. Namely, say $\mathscr{L}$ \emph{computes} the Clifford index if $\cliff(\mathscr L)=\cliff(C)$. Following \cite{ELMS}*[p.~174] and \cite{GL}*{p.~88}, the \emph{Clifford dimension} of $C$ is defined as $\cliffd(C) := \min \{ h^{0}(\mathscr L)-1\, |\, \mathscr L \ \text{computes}\ \cliff(C) \}$.

It is well known that $\cliffd(C)=1$ if and only if $\cliff(C)=\gon(C)-2$, and thus Clifford index and gonality are essentially the same concepts in this case. Rather, if $r:=\cliffd(C)\geq 2$, then $C$ admits a closed immersion $C \hookrightarrow \mathbb{P}^{r}$ \cite{ELMS}*{Lem.~1.1}. The case $r=2$ is known: curves of Clifford dimension equal to 2 are precisely the smooth plane curves of degree $d\geq 5$. Also, $\cliff(C)=\deg(C)-4=\gon(C)-3$ (see \cite{ELMS}*{p.\ 174}). The first equality was recently obtained by Feyzbakhsh-Li in \cite{FL}*{Cor.~5.6} by means of Bridgeland stability and as a particular case of Clifford indeces for higher rank bundles.

The case $r=3$ was addressed by Martens in \cite{Mrt}. He proved that a curve has Clifford dimension $3$ if and only if it is (isomorphic to) a complete intersection of two cubics in $\mathbb{P}^3$. He also proved the restrictions for higher dimension are strong. But he conjectured, for every $r\geq 3$, the existence of curves $C\subset\mathbb{P}^r$ of genus $4r-2$, Clifford index $2r-3$, and Clifford dimension $r$ \cite{GL}*{Problem (3.10)}. 

Eisenbud-Lange-Martens-Schreyer characterize those conditions in \cite{ELMS}*{Thm.~3.6}, prove they are necessary for $r\leq 9$ in \cite{ELMS}*{p.~203}, sufficient for curves lying on certain $K3$ surfaces \cite{ELMS}*{Thm.~4.3}, and then derive the existence of such curves.
%from a result by Morrison \cite[Lem.~4.2]{ELMS}.

Another way of studying this problem is via gonality. In fact, in \cite{CM}*{p. 193} Coppens-Martens deduced that if $r\geq 2$, then  $\cliff(C)=\gon(C)-3$, so this equality characterizes higher Clifford dimension. They call \emph{exceptional} curves satisfying this property. 
As a natural continuation of the study of curves on $K3$ surfaces, Knutsen-Lopez proved in \cite{KL}*{Thm.~1.1}, that there are no exceptional curves on Enriques surfaces other than plane quintics.

The aim of this article is to address the general integral case, so that, in particular, $C$
 can be possibly singular. To begin with, we recall an extended version of Clifford's Theorem proved by Eisenbud-Harris-Koh-Stillman in \cite{EKS}*{App., Thm.~A} and also by Kleiman-Martins in \cite{KM}*{Lem.~3.1.(a)}. Namely, let $\fff$ be a torsion-free sheaf of rank $1$ on $C$ of degree $d$. If $h^0(\fff), h^1(\fff)\geq 1$ then $d-(h^0(\fff)-1)\geq 0$. Moreover, equality also holds for a new class of curves, which are rational, with a unique singularity, for which the maximal ideal agrees with the conductor \cite{EKS}*{Thm.~A.(ii).(c)}. They were called \emph{nearly normal} in \cite{KM}*{Def.~2.15}, as they are the only ones whose canonical model $C'$ is arithmetically normal \cite{KM}*{Thm.~5.10} (see \eqref{cm} for details).

 So a general theory of Clifford index should allow rank $1$ torsion-free sheaves in its computation. In fact, this is the approach adopted by Altman-Kleiman in \cite{AK}*{p.~190} when defining linear series, and, for instance, in \cite{LMS} (and implicitly in \cites{Mt,RSt}) for the definition of gonality, i.e., it is the smallest degree of a torsion-free sheaf $\fff$ of rank $1$ on $C$ with $h^0(\fff)\geq 2$. Within this framework, it is proved in \cite{Mt}*{Thm.~2.1} that $C$ is rational nearly normal iff $\gon(C)=2$, and computed by a $g^1_2$ with an irremovable base point (a notion introduced in \cite{RSt}*{p.~190}). Thus, the Clifford index of an integral curve, if defined in the same spirit, keeps being a measure of how far a curve is from having gonality $2$. 

Hence, for the remainder, given a torsion-free sheaf $\fff$ of rank $1$ on $C$, set its \emph{Clifford index} as $\cliff(\fff)  = \deg(\fff) - 2(h^{0}(\fff)-1)$. The \emph{Clifford index} of $C$ is
$$
     \cliff(C)  =  \min \left\{ \cliff\,(\fff) \,|\, h^{0}(\fff) \geq 2\ \text{and}\ h^{1}(\fff)\geq 2 \right\}.
$$
Say again $\fff$ \emph{computes} the Clifford index of $C$ if $\cliff(\fff)=\cliff(C)$, and, accordingly, define the \emph{Clifford dimension} of $C$ as
$$
\cliffd(C) := \text{min} \{ h^{0}(\fff)-1\, |\, \fff \ \text{computes}\ \cliff(C) \}.
$$
In this setup, we address some of the problems mentioned earlier, extending the discussion to singular curves. We summarize the results that we obtained below.

\medskip

\noindent \textbf{Theorem.} \emph{Let $C$ be an integral and complete curve over an algebraic closed field. 
Assume that $r:=\cliffd(C)\geq 2$. 
We have that:
\begin{itemize}
    \item[(I)] If $\cliff(C)$ is computed by an invertible sheaf, then:
\begin{itemize}
    \item[\rm (i)] there exists a closed immersion $C \hookrightarrow \mathbb{P}^{r}$; in particular, if $r=2$, then $C$ is (isomorphic to) a plane curve;
    \item[\rm (ii)] if $C$ is a plane curve of degree $d\geq 5$, then:
\begin{itemize}
\item[(a)] $\cliff(C)\leq d-4$;
\item[(b)] $\gon(C)\leq d-1$, and equality holds if the gonality is computed by a base-point-free pencil;
\end{itemize}
\end{itemize}
\item[(II)] If $C$ is a plane, monomial, unicuspidal curve of degree $d\geq 5$, then:
\begin{itemize}
    \item[(i)] $\cliffd(C)=2$;
    \item[(ii)] $\cliff(C)=d-4$ and is computed by an invertible sheaf;
    \item[(iii)] $\gon(C)=d-1$ and is computed by both a base point free pencil and by a pencil with an irremovable base point. 
\end{itemize} 
\item[(III)] There's a nonplanar $C$ with a point of multiplicity $\alpha$ and $\cliffd(C)=2$, $\forall\alpha$.
\end{itemize}
}

\medskip

The proof of the above statements is carried out in \eqref{cliffgeral} and \eqref{cliffgthm}. The key ingredients for (II) are: (1) we prove in \eqref{gonality} that both the Clifford index and gonality of a monomial curve can be computed by a sheaf $\fff$ generated by monomial sections; (2) we get in \eqref{h1h1} a formula for $h^{1}(\fff)$ for any such an $\fff$, showing, in particular, if it is eligeble or not to compute the Clfford index; (3) we develop in \eqref{cliffgthm1} a combinatorial method for computing the Clifford index of any such $\fff$. Thus, (1), (2) and (3) provide an algorithm to compute the Clifford index of a monomial curve. We give a simple example of the method in \eqref{exaexc} for a curve which turns out to be exceptional of Clifford dimension $3$. Also, \eqref{gonality} enable us to characterize trigonal monomial curves in \eqref{trigonalchar}.

Surprisingly, (III) yields that the expected result that plane curves (of degree $d\geq 5$) are the ones with Clifford dimension $2$, as in the smooth case, does not hold in general. Indeed, in \eqref{cliffgthm}.(ii) we exhibit a family of curves of Clifford dimension $2$ which are not planar. We calculated their Clifford index as well, and, as predicted by (I), it is not computed by any invertible sheaf. Finally, based in \cite{LMS}*{Thm.~2.2}, we also observe in \eqref{scrollardim} that it's possible to define a \emph{scrollar dimension} of a pencil on $C$ via its canonical model of $C'$ and we show in \eqref{prpscd} how it characterizes Clifford dimension $2$.

There are natural questions related to the present article, of which we highlight some: (1) Does the equivalence $\cliffd(C)\geq 2
\Longleftrightarrow \cliff(C)= \gon(C)-3$ hold for any integral curve? We figure the answer is likely affirmative if the Clifford index is computed by an invertible sheaf, as it seems possible to adjust the techniques used in \cite{ELMS} when so;
%é valida, especialmente por meio do estudo do invariante de Castelnuovo $C(d,g,r)$ que conta o número de $(r-2)$-planos $(2r-2)$-secantes no mergulho de Clifford $C \subset \mathbb{P}^{r}$ de uma curva de gênero $g$ e grau $d$. Tal estudo claramente implicaria na desigualdade desejada;
(2) Assume $\cliff(C)$ is computed by a non-invertible sheaf and $\cliffd(C)=r$; pull it back to $\cb$, remove torsion, and use the resulting pencil to get a curve $\widetilde{C}\subset\mathbb{P}^{r}$. The curve $\widetilde{C}$ stands for a \emph{Clifford model} of $C$, as its construction mimics the one of the \emph{canonical model} $C'$ of a non-Goresntein curve $C$ (see \cite{KM});  since, for instance, $C'$ encodes the gonality of $C$ \cite{LMS}*{Cor.~2.3}, what can be said about $C$ by dint of $\widetilde{C}$? (3) As it happens for Clifford dimension $2$, are there unexpected curves of Clifford dimension $3$ which are not intersection of  two cubics? We hope to address at least part of those problems in a forthcoming work. 

\
   
\noindent{\bf Acknowledgments.} 
This work corresponds to part of the Ph.D. Thesis \cite{At} of the fourth-named author. The third-named author is partially supported by CNPq grant number 308950/2023-2 and FAPESP grant number 2024/15918-8. The fourth-named author is supported by CAPES grant number 88887.821937/2023-00.

\section{Preliminaries}
\label{prelim}

This section provides a brief overview of several concepts that will be used later on, including linear systems, gonality, canonical models, nearly normal curves, semigroup of values, and scrolls.  

To start with, and throughout, $C$ stands for an integral and projective curve of arithmetic genus $g$ defined over an arbitrary algebraically closed field $k$. Let $\oo_C$, or simply $\oo$, dentote the structure sheaf of $C$, and let $\ww$ be its dualizing sheaf. 

Recall that a point $P\in C$ is \emph{Gorenstein} if the stalk $\ww_P$ is a free $\oo_P$-module. The curve $C$ is said to be \emph{Gorenstein} if all of its points are Gorenstein, or equivalently,
if $\ww$ is invertible. Also, recall that $C$ is \emph{hyperelliptic} if there is a double cover $C\to\mathbb{P}^1$. If $C$ is hyperelliptic, it is nessarily Gorenstein \cite{KM}*{Prp.~2.6.(2)}. We will later see that hyperelliptic curves are the only Gorenstein curves with Clifford index zero.  

\begin{sbs}[Linear Series]
 Following \cite{AK}, a \emph{linear series of degree $d$ and dimension $r$} on $C$ (referred as a $g_d^r$ for short) is a set of exact sequences, identified by a pair
$$
%{\rm L}:=
(\fff,V)= \{0\to \mathscr{I} \stackrel{\iota_{x}}{\longrightarrow} \ww \longrightarrow \mathscr{Q}_\lambda \to 0\}_{x\in V\setminus\{0\}}
$$
where $\fff$ is a torsion-free sheaf of rank $1$ on $C$ with $\deg \fff :=\chi (\fff )-\chi (\mathcal{O}) = d$, and $V$ is a nonzero subspace of $H^{0}(\fff)$ of dimension $r+1$. Also, $\mathcal{I} := \mathcal{H}{\rm om}(\fff,\ww)$ and as $x\in V\setminus\{0\}$, it induces an injection $\varphi_x:\oo\hookrightarrow\fff$; so set $\iota_{x}:=\mathcal{H}{\rm om}(\varphi_x,\ww)$. Finally, let $\lambda\in\mathbb{P}(V^*)$ be the point associated to $x$, then $\mathscr{Q}_{\lambda}=\text{coker}(\iota_x)$. 

Note that, if $\oo\subset\fff$, then
\begin{equation}\label{eqdega}
 \deg\fff = \sum_{P\in C} \dim\bigl(\fff_P\big/\op \bigr)
  \end{equation}
In addition, call a point $P\in C$ a \emph{base point of $(\fff,V)$} if,
for all $x\in V$, the injection $\varphi_{x,P} :\op\hookrightarrow\fff_P$ is not an
isomorphism.  Call a base point \emph{removable} if it isn't a base
point of $(\oo\langle V\rangle,V)$, where $\oo\langle V\rangle$ is the $\oo$-submodule of $\fff$ generated by $V$. Say $(\fff,V)$ is \emph{base point free} if it has no base points. If $\fff$ is invertible and it is generated by $V$, then the linear system induces a morphism $ C\to\mathbb{P}^r$. If $(\fff,V)$ is complete, i.e., $V=H^0(\fff)$, then write $(\fff,V)=|\fff|$.
\end{sbs}
\begin{defi}
The \emph{gonality} of $C$ is the minimal $d$ for which there's a $g_d^1$ on $C$, or equivalently, for which there's a torsion-free sheaf $\fff$ of rank $1$ on $C$ of degree $d$ and $h^0(\fff)\geq 2$.  Denote the gonality of $C$ by $\gon(C)$. Let $\fff$ be a torsion-free sheaf of rank 1 on $C$. Say $\fff$ \emph{contributes} to the gonality if $h^0(\fff)\geq 2$, and say $\fff$ \emph{computes} the gonality if, besides, $\deg(\fff)=\gon(C)$.
%Say $\fff$ \emph{contributes} to $\gon(C)$ if $h^0(\fff)\geq 2$. 
\end{defi}

\begin{sbs}[Canonical Model]{\label{cm}} Given a sheaf $\mathscr{G}$ on $C$, if $\varphi :\mathcal{X}\to C$ is a morphism from a scheme $\mathcal{X}$ to $C$, set
$$
\oo_{\mathcal{X}}\mathscr{G}:=\varphi^* \mathscr{G}/\rm{Torsion}(\varphi^*\mathscr{G})
$$
and for each coherent sheaf $\fff$ on $C$ set 
$$
\fff^n:=\rm Sym^n\fff/\rm Torsion (\rm {Sym}^n\fff).
$$ 
%Given a vector space $V \subset k(C)$ we also use the notation 
%\begin{equation} \label{h0omegan}
%    V^{n}:= \left \{\sum_{i=1}^{m} f_{i,1}\ldots f_{i,n} \,\bigg|\, f_{i,j} \in V, \   m\in \nn^{*} \right \}
%\end{equation}}

Consider the normalization map $\pi :\cb\rightarrow C$.
In \cite{R}*{p.\,188} Rosenlicht showed that the linear
system $(\oo_{\cb}\ww,H^0(\ww))$
is base point free. %where $\ww$ stands for the dualizing sheaf on $C$.
He then considered the induced morphism $\psi :\cb\rightarrow{\mathbb{P}}^{g-1}$
and called its image $C':=\psi(\cb)$ the \emph{canonical model} of $C$. Rosenlicht also proved \cite{R}*{Thm.\,17}
that if $C$ is nonhyperelliptic, then the map $\pi :\cb\rightarrow C$ 
factors through a map $C'\rightarrow C$. %Set $\oo':=\pi'_*(\oo_{C'})$ in this case. 
Let $\widehat{C}:=\rm {Proj}(\oplus\,\ww ^n)$ be the blowup of $C$ along $\ww$ and
$\widehat{\pi} :\widehat{C}\rightarrow C$ be the natural morphism. 
%Set $\widehat{\oo}=\widehat{\pi}_*(\oo_{\widehat{C}})$ and $\widehat{\oo}\ww:=\widehat{\pi}_*(\oo _{\widehat{C}}\ww)$. 
In \cite{KM}*{Dfn.\,4.8} one finds
another characterization of the canonical model $C'$, namely, it is the image of the morphism  
$\widehat{\psi}:\widehat{C}\rightarrow{\mathbb{P}}^{g-1}$ defined by the linear system 
$(\oo_{\widehat{C}}\ww,H^0(\ww))$. If $C$ is nonhyperelliptic, then $\widehat{\psi}:\widehat{C}\rightarrow C'$ is an 
isomorphism \cite{KM}*{Thm. 6.4}.

Let $\ccc:=\mathcal{H}\rm {om}(\overline{\oo},\oo)$ be the \emph{conductor sheaf}. Note that, for each $P\in C$, we have the equality  $\ccc_P=(\op:\obp)$, where $\obp$ is the integral closure. That is, the stalk of the conductor sheaf agrees with the local conductor. As in \cite{KM}*{Dfn. 2.15}, call $C$ \emph{nearly normal} if $h^0(\oo/\mathscr{C})=1$. So $C$ has only one singular point, say $P$, and $\ccc_P=\mmp$, where $\mmp$ is the maximal ideal of $\op$. Note that, if $g\geq 2$, then $C$ is necessarily non-Gorenstein.
%the \emph{conductor} of $\overline{\oo}$ into $\oo$, and where the equality makes sense if and only if $C$ is nonhyperelliptic. 

By \cite{KM}*{Thm. 5.10}, $C$ is nearly normal iff $C'$ is arithmetically normal, meaning the homogeneous coordinate ring of $C'$ is normal. We will later see that rational nearly normal curves are the only non-Gorenstein curves with Clifford index zero.  

%\begin{defi} \label{defnng}
%\emph{Let $P\in C$ be any point. Set
%$$
%\eta_P:=\dim(\wwp/\op)\ \ \ \ \ \ \ \ \ \ \ \mu_P:=\dim({\widehat{\oo}_{P}}/\wwp)
%$$
%and also
%$$
%\eta:=\sum_{P\in C}\eta_P\ \ \ \ \ \ \ \ \ \ \mu:=\sum_{P\in C}\mu_P
%$$
%\end{defi}

% \begin{rem}
%\label{remrel} 
%\emph{The importance of the concepts above are summarized below:
%\begin{itemize}
\end{sbs}
\begin{sbs}[Semigroup of Values]{\label{sbssem}} Let $P$ be a \emph{unibranch} point of $C$, that is, there is a unique point $\pb$ in the normalization $\cb$ which lies over $P$. As $\pb$ is smooth, its local ring $\oo_{\cb,\pb}$ is a discrete valuation domain. Let $k(C)$ be the field of rational functions of $C$ (which agrees with the field of fractions of $\oo_{\cb,\pb}$). We have a \emph{valuation map}
$$
v_{\pb}: k(C)^*\longrightarrow\mathbb{Z}
$$
So, for any $x\in k(C)^*$, set
\begin{equation}
\label{equval}
v(x)=v_{P}(x):=v_{\pb}(x)\in\mathbb{Z}
\end{equation}
The \emph{semigroup of values} of $P$ is
$$
\sss=\sss_P:=v(\op ).
$$
The set of \emph{gaps} of $\sss$ is
$$
{\rm G} :=\mathbb{N}\setminus\sss = \{\ell_1,\ldots,\ell_g\}.
$$
The last gap $\gamma:=\ell_g$ is called the \emph{Frobenius number} of $\sss$. Set
\begin{equation}
\label{equaab}
\alpha :={\rm min}(\sss\setminus\{ 0\})\ \ \ \ \text{and}\ \ \ \beta :=\gamma+1.
\end{equation}
Both numbers contain relevant local and geometric information. The value $\alpha$ corresponds to the \emph{multiplicity} of $P\in C$, while $\beta$, called the \emph{conductor} of $\sss$, satisfies
$$
v(\ccc_P)=\{ s\in\sss \,|\, s\geq\beta\}
$$
Similarly, the invariant
$$
\delta =\delta_P :=\#({\rm G}) 
$$
agrees with the \emph{singularity degree} of $P\in C$, that is, 
\begin{equation}
\label{equdlt}
\delta=\dim(\obp/\op).
\end{equation} 
As well know, if $\overline{g}$ is the \emph{geometric genus} of $C$, i.e., the arithmetic genus of $\cb$, then  
\begin{equation}
 \label{equggb}   
g=\overline{g}+\sum_{P\in C}\delta_P
\end{equation}
Assume $C$ is rational and \emph{unicuspidal}, that is, $C$ has only one singular point $P$, and $P$ is unibranch. Then \eqref{equggb} yields $g=\delta_P$, that's why $\delta$ is often called the \emph{genus} of the semigroup $\sss$.
\end{sbs}
\begin{sbs}[Scrolls]{
\label{secscr}}
A \emph{(rational normal) scroll} $S:=S_{m_1 \ldots m_d}\subset\mathbb{P}^{N}$ with invariants $m_1\geq\ldots\geq m_d$,  is a projective variety of dimension $d$ which, after a suitable choice of coordinates, is the set of points $(x_0:\ldots: x_N)\subset\mathbb{P}^N$ such that  the rank of
\begin{equation}
\label{equscr}
\bigg(
\begin{array}{cccc}
x_0 & x_1 & \ldots & x_{m_1-1} \\
x_1 & x_2 & \ldots & x_{m_1}
\end{array}
\begin{array}{c}
\big{|} \\
\big{|}
\end{array}
\begin{array}{ccc}
x_{m_1+1} & \ldots &  x_{m_1+m_2} \\
x_{m_1+2} & \ldots &  x_{m_1+m_2+1}
\end{array}
\begin{array}{c}
\big{|} \\
\big{|}
\end{array}
\begin{array}{c}
\ldots    \\
\ldots  
\end{array}
\begin{array}{c}
\big{|} \\
\big{|}
\end{array}
\begin{array}{cc}
\ldots & x_{m_{1}+ \ldots + m_{l}}  \\
\ldots & x_{m_{1}+ \ldots + m_{l}+1}
\end{array}
\bigg)
\end{equation}
is smaller than 2. 
%Set $m_{i}=0$ for $l < i \leq d$, which corresponds to the $x_{i}$ that do not appear in \eqref{equscr}. 
So, in particular,
\begin{equation}
\label{equnnn}
N=e+d-1
\end{equation}
where $e:=m_1+\ldots+m_d$

Note that $S$ is the disjoint 
%(unless $m_i=0$ for some $i$) 
union of $(d-1)$-planes determined by $d$ points points lying on rational normal curves of degree $m_i$ lying on complementary spaces on $\mathbb{P}^N$. We will refer to any of these $(d-1)$-planes as a \emph{fiber}. So $S$ is smooth if $m_i>0$ for all $i\in\{1,\ldots,d\}$. 
From this geometric description one may see that
\begin{equation}
\label{equdgs}
\deg(S)=e
\end{equation}
The scroll $S$ can also naturally be seen as the image of a projective bundle. In fact,  taking $\mathcal{E}:=\oo_{\mathbb{P}^{1}}(m_1)\oplus\ldots\oplus\oo_{\mathbb{P}^{1}}(m_d)$, one has a birational morphism  
\begin{equation*}
 \mathbb{P}(\mathcal{E})\longrightarrow S\subset\mathbb{P}^{N}
\end{equation*}
defined by $\oo_{\mathbb{P}(\mathcal{E})}(1)$.  The morphism is such that any fiber of $\mathbb{P}(\mathcal{E})\to\mathbb{P}^{1}$ is sent to a fiber of $S$, and it is an isomorphism iff $S$ is smooth (see \cites{EH,Rd,Sc} for more details). 
\end{sbs}

%\begin{equation}
%\label{equkkp}
%{\rm K} :=\{ a\in\mathbb{Z}\ |\ \gamma -a\not\in\sss\}
%\end{equation}
%whose importance will be clear later on. Finally, given ${{\rm I}}$,${{\rm J}} \subset \mathbb{Z}$, set 
%\begin{equation}\label{eqij}
%%    {{\rm I}}-{{\rm J}} := \{a \in \mathbb{Z}  |\ a + {{\rm J}} \subset I \}.
%\end{equation}}

\section{Clifford Dimension}
\label{gongliff}

In this section, we focus primarily on the Clifford dimension of a curve $C$. This concept arises from the definitions of the Clifford index for sheaves and curves. When $C$ is smooth, the discussion is framed in terms of line bundles. However, in the broader context of integral curves, the strong connection to gonality requires a revised formalization.

\begin{defi}
Let $\fff$ be a torsion-free sheaf of rank $1$ on $C$. The \emph{Clifford index} of $\fff$ is defined as follows
$$
     \cliff(\fff)  := \deg(\fff) - 2(h^{0}(\fff)-1).
$$
Accordingly, the \emph{Clifford index} of $C$ is
$$
     \cliff(C)  :=  \min \left\{ \cliff\,(\fff) \,|\, h^{0}(\fff) \geq 2\ \text{and}\ h^{1}(\fff)\geq 2 \right\}.
$$
Say $\fff$ \emph{contributes} to the Clifford index if $h^0(\fff)\geq 2$ and $h^1(\fff)\geq 2$. Say $\fff$ \emph{computes} the Clifford index if, besides, $\cliff(\fff)=\cliff(C)$. The \emph{Clifford dimension} of $C$ is 
$$
\cliffd(C) := \text{min} \{ h^{0}(\fff)-1\, |\, \fff \ \text{computes}\ \cliff(C) \}.
$$
Finally, say $C$ is \emph{exceptional} if $\cliffd(C)\geq 2$
\end{defi}

To study those invariants, we begin by recalling an extended version of Clifford's Theorem. It was proved by Eisenbud-Harris-Koh-Stillman in \cite{EKS}*{Thm.~A} and also by Kleiman-Martins in \cite{KM}*{Lem.~3.1}. For later use, chose to restate it below. 

\begin{clifford}
\label{thmclf}
Let $\fff$ be a torsion free sheaf of rank 1 on $C$ such that $h^0(\fff)\geq 1$ and $h^1(\fff)\geq 1$. Then
\begin{itemize}
    \item[(i)] The following inequality holds
\begin{equation}
\label{equcl0}
2(h^0(\fff)-1)\leq \deg(\fff)
\end{equation}
or, equivalently, by Riemann-Roch,
\begin{equation}
\label{equclf}
h^0(\fff)+h^1(\fff)\leq g+1
\end{equation}
or, in terms of Clifford index,
\begin{equation}
\label{equclf2}
\cliff(\fff)\geq 0
\end{equation}
\item[(ii)] Equality holds in \eqref{equcl0},\eqref{equclf} and \eqref{equclf2} if and only if
\begin{itemize}
\item[(a)] $h^0(\fff)=1$ and $\fff=\oo$, or $h^1(\fff)=1$ and $\fff=\ww$, or
\item[(b)]$h^0(\fff)\geq 2$, $h^1(\fff)\geq 2$, and, either
\begin{itemize}
\item[(b.1)] $C$ is hyperelliptic and $\fff$ a multiple of the $g_2^1$, or
\item[(b.2)] $C$ is rational nearly normal and $\fff=\oo\langle 1,t,\ldots,t^d\rangle$ for $d\leq g-1$.
\end{itemize}
\end{itemize}
\end{itemize}
\end{clifford}

We can derive from the above result the following one. 

\begin{prop} \label{prpgn2}
$\cliff(C)=0$ if and only if $\gon(C)=2$.
\end{prop}

\begin{proof}
 By \eqref{thmclf}.(ii), $\cliff(C)=0$ iff $C$ is either hyperelliptic or rational nearly normal, which holds iff $\gon(C)=2$ by \cite{Mt1}*{Thm.~2.1} and \cite{KM}*{Thm.~3.4}. 
\end{proof}

Therefore, Clifford index makes sense as long as the genus is at least $3$ by \eqref{equclf}. But the following result shows that even the case where $g=3$ may be disregarded. 

\begin{prop} \label{propg3}
Assume $g=3$. Then there is a sheaf on $C$ contributing to the Clifford index if and only if $\gon(C)=2$. 
\end{prop}

\begin{proof}
Sufficiency holds by \eqref{prpgn2}. For necessity, assume $\fff$ on $C$ contributes to the Clifford index. Then $\cliff(C):=g+1-(h^0(\fff)+h^1(\fff))=4-(h^0(\fff)+h^1(\fff))\leq 0$. So $\cliff(C)=0$, thus $\gon(C)=2$ by \eqref{prpgn2}.
\end{proof}

So if $g=3$, the Clifford index of $C$ can be defined if and only if $\gon(C)=2$. But, in this case, we know that $\cliff(C)=0$ by \eqref{prpgn2}, and, as easily follows from the definition, $\cliffd(C)=1$. Therefore, for the remainder, we always assume $g\geq 4$. We first recall that, by \cite{KM2}*{Lem.~3.1}, we have that $\gon(C)\leq g$.

\begin{prop}\label{ap3gonacliff}
The following hold:
\begin{enumerate}
\item[\rm (i)] if ${\gon}(C)<g$ then ${\rm Cliff}(C)\leq{\gon}(C)-2$;
\item[\rm (ii)] if $\gon(C)<g$ and $\cliff(C)=\gon(C)-2$, then $\cliffd(C)=1$;
\item[\rm (iii)] if $\cliffd(C)=1$ then $\cliff(C)=\gon(C)-2$; 
\item[\rm (iv)] if ${\gon}(C)=g$ then $\cliffd(C)\geq 2$;
\item[(v)] if $\gon(C)=3$ then ${\rm Cliff}(C)=1$.
\end{enumerate}
\end{prop}

\begin{proof}
Assume $\fff$ computes $\cliff(C)$, and $\G$ computes $\gon(C)$. As $h^0(\fff)\geq 2$, then: (a)  $\fff$ contributes to $\gon(C)$. On the other hand, $h^0(\G)=2$ by  \cite{KM2}*{Lem.~3.1}. Thus
%\begin{equation}
%\label{equah1F}
$h^1(\G)=h^0(\G)+(g-\deg(\G))-1=1+(g-\gon(C))$
%\end{equation}
and hence: (b) $\G$ contributes to $\cliff(C)$ if and only if $\gon(C)<g$.

If $g>0$, then, by (b), $\cliff(C) \leq \deg(\G)-2= \gon(C)-2$, so (i) holds. Now note that $\cliff(\mathcal{G})=\gon(C)-2$. So if $g>0$ and $\cliff(C)=\gon(C)-2$, then, by (b), $\cliff(\mathcal{G})=\cliff(C)$. Thus $\cliffd(C)=1$ and (ii) follows.

Now assume $\gon(C)=g$. Then, by (a), $\deg(\fff)\geq g$. Thus 
$$
h^0(\fff)=(\deg(\fff)-g)+1+h^1(\fff)\geq 3
$$
and hence $\cliffd(C)\geq 2$. So (iv) holds.

If $\cliffd(C)=1$, then $h^0(\fff)=2$ and $\cliff(C)=\deg(\fff)-2$. Now $\deg(\fff)\geq \deg(\mathcal{G})$ by (a). But $\gon(C)<g$ by (iv); thus (b) yields $\deg(\mathcal{G})-2\geq\deg(\fff)-2$. So $\deg(\fff)=\deg(\mathcal{G})=\gon(C)$. Hence $\cliff(C)=\gon(C)-2$ and (iii) follows.

As $g\geq 4$, item (v) follows directly from (i) and \eqref{prpgn2}.
\end{proof}

The expected equivalence $\cliffd(C) = 1 \Longleftrightarrow \cliff(C) = \gon(C) - 2$ may potentially hold true. This would follow from condition (i) above and Brill-Noether's bound \(\gon(C) \leq \lfloor (g + 3)/2 \rfloor\). To the best of our knowledge, besides smooth curves, this bound was obtained for smoothable curves \cite{KM2}*{Prp.~2.6} and unicuspidal monomial curves \cite{CFMt}*{Thm.~3.(ii)}.

\begin{thm} \label{cliffgeral}
Assume $\cliffd(C)=r\geq 2$, and that the Clifford index and the gonality are computed by invertible sheaves. Then:
\begin{itemize}
    \item[\rm (i)] there is a closed immersion $C \hookrightarrow \mathbb{P}^{r}$;
    \item[\rm (ii)] if $\cliffd(C)=2$, then $C$ is (isomorphic to) a plane curve;
    \item[\rm (iii)] if $C$ is a plane curve of degree $d\geq 5$, then:
\begin{itemize}
\item[(a)] $\cliff(C)\leq d-4$;
\item[(b)] $\gon(C)=d-1$.
\end{itemize}
\end{itemize}
\end{thm}

\begin{proof}
{\bf To prove (i)}, assume $\fff$ computes the Clifford index. Then it is generated by global sections. Indeed, otherwise the subsheaf $\mathscr{U} := \oo_{C} \langle H^{0}(\fff) \rangle \subset \fff$, is such that $\deg(\mathscr{U}) < \deg(\fff)$, $h^{0}(\mathscr{U})=h^{0}(\fff)$, and hence $h^{1}(\mathscr{U}) > h^{1}(\fff) \geq 2$. Thus, $\cliff(C) \leq \cliff(\mathscr{U}) < \cliff(\fff) = \cliff(C)$ which is a contradiction. Therefore, as $\fff$ is invertible and globally generated,  the complete linear system $|\fff|$ induces a morphism, say $\varphi: C \to \mathbb{P}^{r}$. 

Assume $\fff$ is not very ample, or, equivalently, $\varphi$ is not an immersion. Then, either: (1) $\varphi$ does not separate points, or (2) $\varphi$ does not separate tangent lines. We will prove that both statements lead to a contradiction. \par
To do so, for $R \in C$ consider the sheaf $\sM_{\{R\}}$ defined by the exact sequence 
$$
0\longrightarrow \sM_{\{R\}} \longrightarrow \oo \longrightarrow \oo_{R}/\mmr \longrightarrow 0
$$
where $\mmr$ is the maximal ideal of $\oo_R$. That is, $\sM_{\{R\}}$ agrees with $\oo$ outside $R$, and its stalk at $R$ is $\mmr$.

%Let $\varphi: C \rightarrow \mathbb{P}^{r}$ be the morphism induced by $|\fff|$ 
% is not very ample, then either 
If (1) holds, there are $P,Q \in C$, $P \neq Q$, with $H^0(\sM_{\{Q\}}\sM_{\{P\}}\fff)= H^0(\sM_{\{Q\}}\fff)$. So consider the sheaf $\mathcal{G}_{1}:=\sM_{\{Q\}}\sM_{\{P\}}\fff$. On the other hand, if (2) holds, then the natural map $v :H^0(\sM_{\{P\}}\fff)\to H^0(\sM_{\{P\}}\fff/\sM_{\{P\}}^2\fff)$ is not surjective. So, take a vector subspace $V\subset H^0(\sM_{\{P\}}\fff/\sM_{\{P\}}^2\fff)$ of codimension $1$ and containing $v(H^0(\sM_{\{P\}}\fff))$.  Then take a subsheaf $\mathcal{H}\subset\sM_{\{P\}}\fff/\sM_{\{P\}}^2\fff$ with $H^0(\mathcal{H})=V$. Let $\phi: \sM_{\{P\}}\fff\to \sM_{\{P\}}\fff/\sM_{\{P\}}^2\fff$ be the natural map, and consider $\mathcal{G}_2:=\phi^{-1}(\mathcal{H})$.

Now, the proof of \cite{KM}*{Lem.~4.10} yields
$$
h^0(\mathcal{G}_i)=h^0(\fff)-1\text{ and }h^1(\mathcal{G}_i)=h^1(\fff)+1
$$
for $i=1,2$. But note that as $\fff$ is invertible, both $\mathcal{G}_1$ and $\mathcal{G}_2$ are torsion free sheaves of rank $1$. Moreover, $h^{1}(\mathcal{G}_{i}) \geq 3$ and $h^{0}(\mathcal{G}_{i})=r \geq 2$. On the other hand, we have  
$$
\cliff(\mathcal{G}_{i})=g+1-(h^0(\mathcal{G}_i)+h^1(\mathcal{G}_i))=g+1-(h^0(\fff)+h^1(\fff))=\cliff(\fff)=\cliff(C)
$$
which yields $\cliffd(C) \leq r-1$, a contradiction. Thus, $\fff$ is very ample and $\varphi$ is a closed immersion. 

{\bf To prove (ii)}, just note that it is an immediate consequence of (i). 

{\bf To prove (iii).(a)}, let $P \in C$ be a simple point. The pencil of lines through $P$ is a $g_{d}^{1}$. Removing $P$ we get a base-point-free $g_{d-1}^{1}$. So $\gon(C) \leq d-1$. 

Suppose $\gon(C)\leq d-2$. By our assumption, the gonality is computed by an invertible sheaf $\fff$ on $C$ such that $\deg(\fff) =m \leq d-2$ and $h^{0}(\fff) = 2$ \cite{LMS}*{Cor.~2.3}. Then, by \cite{LMS}*{Thm.~2.2}, we have that $|\fff|$ yields an inclusion $C' \subset S\subset\mathbb{P}^{g-1}$, where $S$ is a rational normal scroll of dimension $m-1$.

So let us first find $C'$. Note that $\omega=\mathcal{O}(d-3)$. Indeed, it suffices to show that $\mathcal{O}(d-3)$ is of degree $2g-2$ and has at least $g$ independent global sections. Recall $g=(d-1)(d-2)/2$. Then $\deg(\mathcal{O}(d-3))=d(d-3)=2((d-1)(d-2)/2)-2=2g-2$. On the other hand, consider the exact sequence 
\begin{equation} \label{extseq}
    0 \rightarrow \mathcal{I}_{C}(d-3) \rightarrow \mathcal{O}_{\mathbb{P}^{2}}(d-3) \rightarrow \mathcal{O}(d-3) \rightarrow 0.
\end{equation}
We have $H^{1}(\mathcal{O}_{\mathbb{P}^{2}}(d-3))=0$ and $H^{2}(\mathcal{O}_{\mathbb{P}^{2}}(d-3))=H^{0}(\mathcal{O}_{\mathbb{P}^{2}}(-d))=0$. So the long exact sequence in cohomology of \eqref{extseq} yields $H^{1}(\mathcal{O}(d-3))=H^{2}(\mathcal{I}_{C}(d-3))$. As $C$ is a divisor in $\mathbb{P}^{2}$, 
$\mathcal{I}_{C}=\mathcal{O}_{\mathbb{P}^{2}}(-d)$; so $H^{2}(\mathcal{I}_{C}(d-3))=H^{2}(\mathcal{O}_{\mathbb{P}^{2}}(-3))=H^{0}(\mathcal{O}_{\mathbb{P}^{2}})=k$. Thus $h^{1}(\mathcal{O}(d-3))=1$; $h^{0}(\mathcal{O}(d-3))=2g-2+1-g+1=g$, and the claim holds. 

Therefore, writing $\mathbb{P}^{2}=\left\{(X:Y:Z)\right\}$ and taking $x=X/Z$, $y=Y/Z$, we have that the canonical model of $C$ is 
\begin{equation} \label{canmod}
C'=(1:x:x^{2}: \ldots : x^{d-3}: xy: \ldots \ldots :xy^{d-4}: y: \ldots :y^{d-3}) \subset \mathbb{P}^{g-1}
\end{equation}
for $(1:x:y) \in C$.

Now, $C'\subset S\subset \mathbb{P}^{g-1}$ where $S$ is the scroll induced by $|\fff|$. The fibers of $S$ are $(m-2)$-planes in $\mathbb{P}^{g-1}$ because $S$ is of dimension $m-1$. A generic fiber of $S$ contains $m$ distinct points of $C'$ because, since $\fff$ is invertible, $|\fff|$ is a base point free $g_m^1$, and the fibers of $S$ cut out this $g^1_m$. Say those points are $P_1,\ldots,P_m$, and say they correspond, for $1\leq i\leq m$, to points $(1:x_i:y_i)\in C\subset \mathbb{P}^2$. So, by genericity, we may assume $x_i,y_i\neq 0$ for all $i$, and that the $x_i/y_i$ are all distinct for $1\leq i\leq m$. Consider the vectors $v_i:=(1,x_i/y_i.(x_i/y_i)^2,\ldots,(x_i/y_{i})^{d-3})$ for $1\leq i\leq m$.  They are distinct and form a Vandermonde matrix $m \times (d-2)$. But $m\leq d-2$, so the $v_i$ are linearly independent and, hence, so are the $P_i$, viewed as vectors in $k^{g}$. But this is impossible as the $P_i$ lie in an $(m-2)$-plane of $\mathbb{P}^{g-1}$. Thus $\gon(C)=d-1$.

{\bf To prove (iii).(b)} consider the exact sequence
\begin{equation} \label{extseq2}
    0 \rightarrow \mathcal{I}_{C}(1) \rightarrow \mathcal{O}_{\mathbb{P}^{2}}(1) \rightarrow \mathcal{O}(1) \rightarrow 0.
\end{equation}
As $\mathcal{I}_C=\oo_{\mathbb{P}^2}(-d)$ and $d\geq 5$, it follows that $H^0(\mathcal{I}_C(1))=0$. Also, we have that $H^{1}(\mathcal{I}_{C}(1))=H^{1}(\mathcal{O}_{\mathbb{P}^{2}}(-d+1))=0$. Thus the long exact sequence in cohomology of \eqref{extseq2} yields $H^0(\oo_{\mathbb{P}^{2}}(1))\cong H^0(\oo(1))$. So $h^0(\oo(1))=3$. Also, $\deg(\mathcal{O}(1))=d$. Thus $\cliff(\mathcal{O}(1))=d-2(3-1)=d-4$. So we have to check whether this sheaf contributes to the Clifford index. Indeed, by Riemann-Roch, we get
\begin{align*}
h^{1}(\oo(1))&=h^0(\oo(1))-\deg(\oo(1))-1+g \\
&=3-d-1+\dfrac{(d-1)(d-2)}{2} = \dfrac{(d-2)(d-3)}{2} \geq 2
\end{align*}
iff $d=0$ (precluded) or $d \geq 5$. Therefore, $\cliff(C) \leq \cliff(\mathcal{O}(1))=d-4$. In particular, $\cliff(C) \leq d-4$ under our hypothesis of the Clifford index being computed by an invertible sheaf. We are done.
\end{proof}
\begin{sbs}[Scrollar Dimension]{\label{scrollardim}}
Let $\fff$ be a torsion free sheaf of rank 1 on $C$ such that $h^{0}(\fff) \geq 2$. Let $V\subset H^0(\fff)$ be a subspace such that $\dim(V)=2$. By \cite{LMS}*{Thm.~2.2}, the pencil $(\fff,V)$ induces an inclusion $C'\subset S_{(\fff,V)}\subset\mathbb{P}^{g-1}$, where $C'$ is the canonical model and $S_{(\fff,V)}$ is a rational normal scroll. Also, from the proof of \cite{LMS}*{Thm.~2.2.(I)}, we get that $\dim(S_{(\fff,V)})=\deg(\fff)-(h^0(\fff)-1)$. In particular, the scroll $S_{(\fff,V)}$ depends on $V$, but its dimension does not. So we may define the \emph{scrollar dimension} of $\fff$, denoted by ${\rm scd}(\fff)$, as $\dim(S_{(\fff,V)})$, whatever is the $2$-dimensional vector space $V\subset H^0(\fff)$. Therefore, 
\begin{equation} \label{dimscroll}
    {\rm scd}(\fff) = \deg(\fff)-(h^0(\fff)-1).
\end{equation}
As an immediate consequence of \eqref{dimscroll}, we have the following relation
\begin{equation}
\label{equsd1}
    \cliff(\fff) = {\rm scd}(\fff) - (h^{0}(\fff)-1).
\end{equation}
We know that the minimal dimension of a scroll containing $C'$ is $\gon(C)-1$ owing to \cite{KM2}*{Thm.~5.2}. It is the scrollar dimension of any sheaf computing gonality. In the following result, we characterize curves admitting other sheaves with such a property under certain hypotheses, which, as stressed in the Introduction, hold for any smooth curve. 

\begin{thm}
\label{prpscd}
Assume $\gon(C)<g$ and $\cliff(C)\geq\gon(C)-3$. Then there is a sheaf on $C$ with minimal scrollar dimension but not computing its gonality if and only if $\cliffd(C)=2$.
\end{thm}

\begin{proof}
Let $\fff$ be a torsion-free sheaf of rank $1$ on $C$ such that ${\rm scd}(\fff)=\gon(C)-1$ and $\deg(\fff)>\gon(C)$. Then \eqref{dimscroll} yields 
\begin{equation}
\label{equsd2}
h^0(\fff)= 2+(\deg(\fff)-\gon(C))\geq 3
\end{equation}
On the other hand, 
\begin{align*}
h^1(\fff)&=h^0(\fff)-\deg(\fff)+g-1\\
         &=(2+(\deg(\fff)-\gon(C)))-\deg(\fff)+g-1\\
         &=1+(g-\gon(C))\geq 2
\end{align*}
Therefore, $\fff$ contributes to the Clifford index. Now \eqref{equsd1} and \eqref{equsd2} yield
\begin{equation}
\label{equsd3}
\cliff(\fff)=(\gon(C)-1)-(h^{0}(\fff)-1)\leq \gon(C)-3
\end{equation}
But $\cliff(C)\geq\gon(C)-3$. Thus $\cliff(\fff)=\cliff(C)=\gon(C)-3$, and hence the equality in \eqref{equsd3} yields $h^0(\fff)=3$. So $\cliffd(C)\leq 2$. But if $\cliffd(C)=1$, then $\cliff(C)=\gon(C)-2$ by \eqref{ap3gonacliff}.(iii), which is precluded. It follows that $\cliffd(C)=2$ as desired.

Conversely, assume $\cliffd(C)=2$. Then \eqref{ap3gonacliff}.(ii) yields $\cliff(C)=\gon(C)-3$. Say $\fff$ computes the Clifford index. Then $h^0(\fff)=3$ and $\cliff(\fff)=\gon(C)-3$. So \eqref{equsd1} yields ${\rm scd}(\fff) = \gon(C)-1$. But \eqref{dimscroll} yields $\deg(\fff)=\gon(C)+1$, so $\fff$ does not compute the gonality. We are done.
\end{proof}
\end{sbs}

\section{Clifford Index of Monomial Curves}

In this section, we introduce monomial curves, which will be studied in detail. We aim to obtain combinatorial formulae that compute their Clifford index in terms of the semigroup of values of their singularities.

To begin with, a curve $C$ is said \textit{monomial}, if it is the image of a map
\begin{gather*}
\begin{matrix}
\mathbb{P}^1 & \longrightarrow & \mathbb{P}^m  \\
(s:t)  & \longmapsto     & (s^{n_m}:s^{n_m-n_1}t^{n_{1}}: \cdots : s^{n_m-n_{m-1}}t^{n_{m-1}}:t^{n_{m}})
\end{matrix}
\end{gather*} 
which, for short, we write
$$
C = (1:t^{n_{1}}: \cdots : t^{n_{m-1}}:t^{n_{m}}).
$$
If $n_1\geq 2$, then $P=(1:0:\ldots:0)$ is a singular point of $C$. For the remainder, we assume $C$ is \emph{unicuspidal}, that is, it has only one singularity (which is a cusp). Also, we assume $P$ is this point. Then $Q:=(0:\ldots:0:1)$, which is the point of $C$ at infinity, is regular, which happens iff $n_{m}=n_{m-1}+1$. Now recall \eqref{sbssem}.   Let $\sss$ be the semigroup of $P$, let ${\rm G}$ the set of gaps of $\sss$, and $\gamma$ be its Frobenius number. As $C$ is assumed to be unicuspidal, then we will refer to $\sss$ as the \emph{semigroup} of $C$

By \cite{LMS}*{Thm.~4.1}, the dualizing sheaf $\ww$ of $C$ admits an embedding on the constant sheaf of rational functions $\mathscr{K}$, such that 
\begin{equation}\label{eqduash}
H^{0}(\omega)=\{t^i\,|\,i\in\gamma-\gap\}
\end{equation}
Therefore, the canonical model of $C$ is
\begin{equation}\label{eqcanmod}
C'=(1:t^{b_2}:\ldots :t^{b_{\delta}})
\end{equation}
where $\{0,b_2,\ldots,b_{\delta}\}=\gamma-\gap:=\{\gamma-\ell\,|\,\ell\in\gap\}$.

We now fix some notation. Following \cite{BF}*{p.~420}, set 
\begin{equation}
\label{equkkk}
\kk:=\{a\in\mathbb{Z}\,|\,\gamma-a\not\in\sss\}
\end{equation}
referred as the \emph{standard canonical relative ideal} of $\sss$. Let $v$ be the valuation map at $P$ defined in \eqref{equval}. By \cite{St}*{Thm~2.11}, we have
\begin{equation}
\label{equvpw}
v(\ww_{P})=\kk   
\end{equation} 
In the case of a unicuspidal monomial curve, \eqref{equvpw} can also be seen by \eqref{eqduash} and the fact that $\ww$ is generated by global sections. Following again \cite{BF}*{p.~420}, given $A,B\subset\mathbb{Z}$, set 
%$A+B=\{a+b\,|\, a\in A,\ b\in B\}$.  
$$
A-B:=\{z\in\mathbb{Z}\,|\, z+B\subset A\}.
$$
Finally, given $f_1,\ldots,f_n\in k(C)$, let $\fff:=\oo\langle f_1,\ldots,f_n\rangle$ be the subsheaf of $\mathscr{K}$ generated by the $f_i$. In particular, the stalks of $\fff$ are
$$
\fff_R=\oo_R\langle f_1,\ldots,f_n\rangle:=f_1\oo_R+\ldots +f_n\oo_R
$$ 
for every $R\in C$. 

Later on in \eqref{gonality}, we will see that the Clifford dimension (and consequently the Clifford index) of a unicuspidal monomial curve can be computed using a sheaf generated by monomial sections. These sheaves play a crucial role in our analysis, and we aim to derive some results regarding them. We begin with the following.

\begin{lem} \label{h1h1}
Let $C$ be a unicuspidal monomial curve with semigroup $\sss$, set of gaps $\gap$, and Frobenius number $\gamma$. 
Let also $\fff:=\oo\langle 1,t^{a_1},\ldots, t^{a_n}\rangle$. Consider the set $\gap':=\{\ell\in\gap\, |\, \ell>a_n\ \text{and}\ \ell\notin \bigcup _{i=1}^n (a_i+\sss)\}$. Then
\begin{equation} \label{eqh1h1}
     {\rm Hom}(\fff,\ww)= \langle t^i\, |\, i \in \gamma -\gap' \rangle.
\end{equation}
In particular, 
\begin{equation} \label{equh1f}
     h^1(\fff)= \#(\gap')
\end{equation}
and $\fff$ contributes to the Clifford index if and only if $\#(\gap') \geq 2$. 
\end{lem}

\begin{proof}
Set $\mathscr{I}:= \mathcal{H}\text{om}(\fff,\omega)$. Then $H^{0}(\mathscr{I})={\rm Hom}(\fff,\ww)$. Note that $\mathscr{I}_{R}=(\omega_{R}:\fff_{R})$ for every $R \in C$. By \eqref{eqduash}, we have $\ww_R=\oo_R+t^{b_2}\oo_R+\ldots +t^{b_{\delta}}\oo_R$,
where $\{0,b_2,\ldots,b_{\delta}\}=\gamma-\gap$, for every $R\in C$, because $\ww$ is generated by global sections. On the other hand, $\fff_R=\oo_R+t^{a_1}\oo_R+\ldots t^{a_n}\oo_R$, for every $R\in C$, by the very definition of $\fff$. Therefore, if $R \neq P,Q$, then $\mathscr{I}_{R}=(\oo_{R}:\oo_{R})=\oo_{R}$. For $Q$, we have $\mathscr{I}_{Q}=(t^{\gamma-1}\oo_{Q}:t^{a_{n}}\oo_{Q})=t^{\gamma-1-a_n}\oo_{Q}$. For $P$, set ${\rm F}:=v(\fff_{P})=\sss\bigcup\big(\cup _{i=1}^n (a_i+\sss)\big)$. Thus, using \eqref{equvpw}, we have that
$$
\mathscr{I}_P =\oo_{P} \langle\, t^{i}\ |\ i \in {{\rm K}}-{{\rm F}} \rangle.
$$ 
As $H^0(\mathscr{I})=\cap_{R\in C}\mathscr{I}_R$, it follows that 
\begin{equation} \label{eqh1f}
    H^0(\mathscr{I})=\langle\, t^{i}\ |\ i \in ({{\rm K}}-{{\rm F}}) \cap[0,\gamma-1-a_{n}] \rangle.
\end{equation}
So we have to prove that
$$
({{\rm K}}-{{\rm F}}) \cap[0,\gamma-1-a_{n}]=\gamma-\big((\gap\setminus{\rm F})\cap[a_n+1,\gamma]\big)
$$ 
First note that $[0,\gamma-1-a_{n}]=\gamma-[a_{n}+1,\gamma]$. Now $\gap\setminus{\rm F}=\mathbb{N}\setminus{\rm F}$ because $\sss\subset{\rm F}$. So it remains to prove that ${{\rm K}}-{{\rm F}}=\gamma-(\mathbb{N} \setminus {{\rm F}})$. To prove ``$\subset$", write $a \in {{\rm K}}-{{\rm F}}$ as $a=\gamma -b$. We have to show that $b \notin {{\rm F}}$. As $\gamma-b \in {{\rm K}}-{{\rm F}}$, by \eqref{equkkk} we have
\begin{equation}
\label{equggb1}
\gamma-((\gamma-b)+f) \notin \sss \ \ \ \ \forall f \in {{\rm F}}.
\end{equation}
So if $b \in {\rm F}$, taking $f=b$ in \eqref{equggb1}, yields $0=b-b \notin \sss$, a contradiction. 

To prove ``$\supset$", suppose that $a \notin {{\rm F}}$. We have to show that $\gamma-a \in {{\rm K}}-{{\rm F}}$. If $\gamma-a \not\in {{\rm K}}-{{\rm F}}$, then there exists $f \in {{\rm F}}$ such that $\gamma-a+f \notin {{\rm K}}$. So, $\gamma-(\gamma-a+f) \in \sss$, which implies $a-f \in \sss$ and thus $a \in f+\sss \subset {\rm F}$, a contradiction. 
\end{proof}

The formula \eqref{equh1f} will be useful for calculating the Clifford index of a unicuspidal monomial curve, which we will do shortly. Additionally, an explicit expression such as \eqref{eqh1h1} is helpful for determining the scroll associated with any linear subpencil of  $|\fff|$, as exploerd in the following example. We illustrate that the scrollar dimension of $\fff$, as defined in \eqref{scrollardim}, is independent of the choice of subpencil.

\begin{exam}
\label{exedif}
Consider the curve
$$
C=(1:t^5:t^6:t^{10}:t^{11}:t^{12}:t^{15}:t^{16}:t^{17}:t^{18})\subset\mathbb{P}^{9}
$$
It is a rational unicuspidal monomial curve. The semigroup of $C$ is 
$$
\sss=\{ 0,5,6,10,11,12,15,16,17,18,20, \to \}.
$$
So $g=\delta_P=10$. Note $\sss={{\rm K}}$, so $\wwp=\op$, and hence $C$ is Gorenstein. By \eqref{eqduash},
$$
H^{0}(\omega)=\langle 1,t^5,t^6,t^{10},t^{11},t^{12},t^{15},t^{16},t^{17},t^{18} \rangle
$$
Now, consider the sheaf $\fff=\oo_C\langle 1, t^4, t^5, t^6\rangle$. We will describe two different scrolls induced by $(\fff,V)$ for two different spaces $V$.  Set $\mathscr{I}:=\mathcal{H}\text{om}(\fff,\ww)$. To use \eqref{eqh1f}, we compute $v(\fff_P)$. Note $\fff_P=\op+t^4\op$. 
Thus $v(\fff_P)=\sss\cup(\sss+4)=\sss\cup\{1,9,14,19\}$, which yields  $\gamma-\gap'=19-\{7,8,13\}=\{6,11,12\}$. Hence \eqref{eqh1f} yields
$$
H^0(\mathscr{I})={\rm Hom}(\fff,\ww)=\langle t^6,t^{11},t^{12} \rangle
$$

First, take $V:=\langle 1, t^{5}\rangle\subset H^{0}(\fff)$, and consider the multiplication map
\begin{gather*}
\begin{matrix}
\varphi: V \otimes H^{0}(\mathscr{I}) & \longrightarrow & H^{0}(\omega)   \\
t^{i} \otimes t^{j} & \longmapsto & t^{i+j}
\end{matrix}
\end{gather*}
The scroll defined by $(\fff,V)$ is given by $\det_{2}(A_{\varphi})$, where $A_{\varphi}$ is the $2 \times h^{0}(\mathscr{I})$ matrix defined as $(A_{\varphi})_{i,j}=t^{i+j}$. Now write
$$
\mathbb{P}^9=\mathbb{P}(H^{0}(\omega)) =  \{(1: t^{5}: \ldots :t^{18})\} = \{(x_{0}: \ldots :x_{9})\}
$$
Thus $S:=S_{(\fff,V)}$ is the scroll on $\mathbb{P}^{9}$ cut out by the $2 \times 2$ minors of 
$$
A_{\varphi}=\bigg(
\begin{array}{ccc}
x_2 & x_{4} & x_{5} \\
x_{4} & x_{7} & x_{8}
\end{array}
\bigg)
$$
Hence, by \eqref{equscr}, $S=S_{2,1,0,0,0,0,0}$, and $\dim S = 7$. Clearly, $C'=C\subset S$.

Now, consider $V = \langle t^4,t^6 \rangle$. Then 
$$
A_{\varphi}=\bigg(
\begin{array}{ccc}
x_3 & x_{6} & x_{7} \\
x_{4} & x_{8} & x_{9}
\end{array}
\bigg),
$$
$S=S_{1,1,1,0,0,0,0}$ and $\dim S = 7$. 

Let us now compute the degree of $\fff$ to check \eqref{dimscroll}. We will use \eqref{eqdega}. Set $\deg_R(\fff):=\dim(\fff_R/\oo_R)$. We have $\fff_R=\oo_R$ for every $R\neq P,Q$, so $\deg_R(\fff)=0$. Also, $\fff_Q=t^6\oo_Q$ so $\deg_Q(\fff)=6$. Finally, $\deg_{P}(\fff)=\#(v(\fff_P)\setminus\sss)$ owing to \cite{BF}*{p.~438}. But $v(\fff_P)=\sss\cup\{1,9,14,19\}$ as computed earlier. So $\deg_P(\fff)=4$. Hence $\deg(\fff)=6+4=10$, and \eqref{dimscroll} yields
$$
{\rm scd}(\fff) = \deg(\fff)-h^0(\fff)+1=10-4+1=7
$$

Now we consider a non-Gorenstein curve. In such a case, though $C'\not\cong C$, the computation is similar. Set 
$$
C=(1:t^3:t^6:t^9:t^{10}:t^{12}:t^{13}:t^{14})\subset\mathbb{P}^7
$$
The semigroup of $C$ is 
$$
\sss=\{ 0,3,6,9,10,12,13,14 \to \}.
$$
So its genus is $g=\delta_P=7$. Note that $\sss \neq {{\rm K}}$, so $C$ is non-Gorenstein. By \eqref{eqduash},
$$
H^{0}(\omega)=\langle 1,t^3,t^4,t^{6},t^{7},t^{9},t^{10} \rangle
$$
Now, consider the sheaf $\fff=\oo_C\langle 1, t^3, t^4, t^6\rangle$. Again, we will describe two different scrolls induced by $(\fff,V)$ for two different spaces $V$.  By \eqref{eqh1f}, one can check that
$$
H^0(\mathscr{I})=\langle 1, t^{3} \rangle
$$
First, take $V:=\langle 1, t^{4}\rangle\subset H^{0}(\fff)$, and write  
$$
\mathbb{P}^6=\mathbb{P}(H^{0}(\omega)) =  \{(1: t^{3}: \ldots :t^{10})\} = \{(x_{0}: \ldots :x_{6})\}.
$$
The scroll $S$ defined by $(\fff,V)$ is given by the $2 \times 2$ minors of 
$$
A_{\varphi}=\bigg(
\begin{array}{cc}
x_0 & x_{1} \\
x_{2} & x_{4}
\end{array}
\bigg)
$$
Thus, by \eqref{equscr}, $S=S_{1,1,0,0,0}$, and $\dim S = 5$.

Now, consider $V = \langle t^3,t^6 \rangle$. Then, 
$$
A_{\varphi}=\bigg(
\begin{array}{cc}
x_1 & x_{3} \\
x_{3} & x_{5}
\end{array}
\bigg),
$$
$S=S_{2,0,0,0,0}$ and $\dim S = 5$. In fact, by \eqref{dimscroll}, we have that
$$
{\rm scd}(\fff) = \deg(\fff)-h^0(\fff)+1=8-4+1=5
$$
\end{exam}

% Now we study the Clifford index of unicuspidal monomial curves. We start by the following result.

In the following result, we present a formula for the Clifford index of sheaves generated by monomial sections. Combining \eqref{h1h1}, \eqref{cliffgthm1}  and \eqref{gonality} ahead we get a method to compute the Clifford index of any unicuspidal monomial curve. 

\begin{lem} \label{cliffgthm1}
Let $C$ be a unicuspidal monomial curve with semigroup $\sss$,  set of gaps $\gap$, and Frobenius number $\gamma$. Let also $\fff:=\oo\langle 1,t^{a_1},\ldots, t^{a_n}\rangle$. Consider the set ${\rm E}:=\bigcup_{i=1}^n(a_i+\sss)$. Then
\begin{equation} 
\label{cliffg}
     \cliff(\fff)=\#\big((\gap\setminus{\rm E})\cap[1,a_n]\big)-\#\big(\sss\cap[1,a_n]\big)+\#\big(({\rm E}\setminus\sss)\cap[a_n+1,\gamma]\big).
\end{equation}
If $\fff$  is invertible, then 
\begin{equation} \label{cliffg2}
     \cliff(\fff)=\#\big(\gap\cap[1,a_{n}]\big)-\#\big(\sss\cap[1,a_{n}]\big)
\end{equation}
\end{lem}

\begin{proof}
Let $\gap'$ be as in \eqref{h1h1}. Note that $\gap'=(\gap \setminus {\rm E}) \cap [a_n+1,\gamma]$. Also, by \eqref{equh1f}, we have that $h^1(\fff)=\#(\gap')$. On the other hand, if $R \neq P,Q$, then $\fff_{R}=\oo_{R}$; for $Q$, we have $\fff_{Q}=t^{a_{n}}\oo_{Q}$; and for $P$,  set ${\rm I}:=\big(\sss\cup{\rm E}\big)\cap[0,\gamma]$, and note that we have $\fff_P=\big(\oplus_{i\in {\rm I}} kt^i\big)\oplus t^{\gamma+1}\obp$. 
As $H^0(\fff)=\cap_{R\in C}\fff_R$, then 
\begin{equation}
\label{equh0f}
H^0(\fff)=\langle t^i\,|\, i \in \big(\sss\cup{\rm E}\big)\cap[0,a_n]\rangle .   
\end{equation} 
Therefore
\begin{align}
    \label{equcl1}
     \cliff(\fff) &=\deg(\fff)-2(h^{0}(\fff)-1) =(g-h^{1}(\fff))-(h^{0}(\fff)-1) \nonumber \\
     &=\#\big(\gap\setminus\gap'\big)-\#\big((\sss\cup{\rm E})\cap[1,a_n]\big).
\end{align}
Now
\begin{align*}
\gap\setminus\gap'=\bigg(\big((\gap \setminus {\rm E}) \cup ({\rm E} \setminus \sss)\big) \cap [1,a_{n}] \bigg) \bigcup \bigg(({\rm E}\setminus \sss) \cap [a_{n}+1, \gamma] \bigg).
\end{align*}
and hence 
\begin{equation}
\label{equcl2}
\#(\gap \setminus \gap')=\#\big( (\gap \setminus {\rm E}) \cap [1,a_{n}] \big)+\#\big (({\rm E} \setminus \sss) \cap [1,a_{n}] \big)+\# \big(({\rm E}\setminus \sss) \cap [a_{n}+1, \gamma] \big).
\end{equation}
On the other hand, 
$$
\big(\sss \cup {\rm E}\big)\cap[1,a_n]=\big(({\rm E} \setminus \sss) \cup \sss \big) \cap [1,a_{n}].
$$
and hence 
\begin{equation}
\label{equcl3}
\#\big((\sss \cup {\rm E})\cap[1,a_n]\big)=\#\big(({\rm E} \setminus \sss) \cap [1,a_{n}]\big)+ \#(\sss \cap [1,a_{n}])
\end{equation}
So \eqref{cliffg} follows from \eqref{equcl1}, \eqref{equcl2} and \eqref{equcl3}. 

To prove  \eqref{cliffg2}, first note that $\oo_P\subset\fff_P\subset \obp$. Now ${\rm F}:=v(\fff_P)=\sss\cup{\rm E}$, hence $\sss\subset {\rm F}\subset \mathbb{N}$. But as $\fff$ is invertible, then $\ff$ is a translation of $\sss$. But this cannot happen unless $\ff=\sss$, which yields ${\rm E}\subset\sss$. Thus, \eqref{cliffg} turns into \eqref{cliffg2}.
\end{proof}

\section{Clifford Dimension of Monomial Curves}

In this section, we study the Clifford dimension of unicuspidal monomial curves. We start by showing that this invariant can always be computed by a sheaf generated by monomial sections. The proof can be easily adapted to show that the same property applies to gonality, leading to a characterization of trigonal monomial curves. Next, we apply the formulae obtained in the previous section to calculate the Clifford dimension of plane curves. Finally, we provide a simple example to illustrate how the method works for a curve of Clifford dimension of 3.

\begin{lem} \label{gonality}
Let $C$ be a unicuspidal monomial curve with $\cliffd(C)=r$. Then,
\begin{itemize}
    \item[(i)] $\cliffd(C)$ is computed by a sheaf of the form $\oo_C\langle 1,t^{a_1}, \ldots, t^{a_{r}}\rangle$ for $a_i\in\nn^*$;
    \item[(ii)] $\gon(C)$ is computed by a sheaf of the form $\oo_C\langle 1,t^{a}\rangle$ for $a\in\nn^*$.
\end{itemize}
\end{lem}

\begin{proof}
Assume $\fff$ computes the Clifford dimension of $C$. In particular, it computes the Clifford index of $C$. Thus it is generated by global sections as seen in the proof of \eqref{cliffgeral}. Now $\fff$ has $r+1$ independent global sections. But any torsion free sheaf of rank $1$ with a global section can be embedded in the constant sheaf of rational functions $\mathcal{K}$ so that
\begin{equation}
\label{equtfs}
\oo \subset \fff \subset \mathcal{K}.
\end{equation}
Thus we may write $\fff=\mathcal{O}_C\langle1,z_1,\ldots,z_r\rangle$ with $z_i=t^{a_i} f_i/h_i$, and $f_i,h_i\in k[t]$ with no common factor, and not having zero as a root either. We may first assume that $a_i\geq 0$ for any $i\in\{1,\ldots,r\}$, i.e, the $z_i$ do not have a pole on $P$. Indeed, if not, as $\deg(x\fff)=\deg(\fff)$ for any $x\in k(C)$, replace $\fff$ by $t^{-{a_i}}\fff$. We may further assume $a_i>0$ for every $i\in\{1,\ldots,r\}$, just replacing $z_i$ by $z_i-z_i(0)$ if necessary. Therefore, the Clifford dimension of $C$ can be computed by a sheaf of the form $\fff=\mathcal{O}_C\langle1,t^{a_{1}}f_{1}/h_{1}, \ldots t^{a_{r}}f_{r}/h_{r} \rangle$, with $f_{i},h_{i}\in k[t]$, ${\rm gcd}(f_{i},h_{i})=1$, $f_{i}(0)\neq0\neq h_{i}(0)$, and $a_{i}>0$. 

Set $\mathcal{G}:=\mathcal{O}_C\langle1,t^{a_{1}}, \ldots, t^{a_r} \rangle$. We will prove that $\deg(\mathcal{G})\leq\deg(\fff)$. As always, let $P$ be the singularity, and $Q$ the point at infinity. Set $U:=C\setminus\{P\}$.  We have that 
\begin{align*}
    \deg_U(\fff) & =\deg_{U\setminus\{Q\}}(\fff)+\deg_Q(\fff) \\
    & = \sum_{i=1}^{r} \deg(h_{i})+\max_{1\leq i \leq r}\{0, a_{i}+\deg(f_{i})-\deg(h_{i}) \} 
    \end{align*}
If the second term of the right hand side of the equality vanishes, then $\deg(h_i)\geq a_i$ for every $1 \leq i \leq r$. Hence $\deg_U(\fff)=\sum_{i=1}^{r} \deg(h_{i}) \geq \max_{i}\{a_{i}\}=\deg_U(\mathcal{G})$. Otherwise, say $a_{s}+\deg(f_{s})-\deg(h_{s}) \geq a_{j}+\deg(f_{j})-\deg(h_{j})$ for all $j$. Fix $j$. Thus $a_{s}+\deg(f_{s})+\deg(h_{j}) \geq a_{j}+\deg(f_{j})+\deg(h_{s})$.  Then we have
\begin{align*}
    \deg_U(\fff) & = \sum_{i=1}^{r} \deg(h_{i})+a_{s}+\deg(f_{s})-\deg(h_{s}) = \sum_{i \neq s} \deg(h_{i})+a_{s}+\deg(f_{s}) \\
    & = \sum_{i \neq s,j} \deg(h_{i})+a_{s}+\deg(f_{s})+\deg(h_{j}) \\
    & \geq \sum_{i \neq s,j} \deg(h_{i})+a_{j}+\deg(f_{j})+\deg(h_{s}) \geq a_{j}
    \end{align*}
Therefore, $\deg_{U}(\fff) \geq \max_{i}\{ a_{i}\}= \deg_{U}(\mathcal{G})$. On the other hand,
\begin{align*}
    \deg_{P}(\fff) & =\#(v(\fff_P)\setminus\sss) \\
    & \geq \#\big(\bigcup_{i=1}^{r}(\sss+a_i)\setminus\sss\big) = \#(v(\mathcal{G}_P)\setminus\sss) = \deg_P(\mathcal{G})
\end{align*}
where monomiality is used in the second equality above. Now, by construction, $h^{0}(\mathcal{G}) \geq r+1 = h^{0}(\fff)$. On the other hand, we proved that $\deg(\mathcal{G})\leq \deg(\fff)$; so this implies that $h^1(\mathcal{G}) \geq h^{1}(\fff))\geq 2$.  Hence $\mathcal{G}$ contributes to the Clifford index. But $\cliff(\mathcal{G})=\deg(\mathcal{G})-2(h^0(\mathcal{G})-1)) \leq \cliff(\fff)$. So $\mathcal{G}$ computes the Clifford index as $\fff$ does. So (i) follows. To prove (ii), just repeat the proof above with $r=1$.
\end{proof}

\begin{exam}
Consider the non-monomial curve $C=(1-t: t^2: t^4:t^5)\subset\mathbb{P}^3$. It has a unique singularity at $P=(1:0:0:0)$, which is a cusp, so $C$ is unicuspidal. The local ring at $P$ is
$$
\oo_P=k\oplus k\frac{t^2}{1-t}\oplus k\frac{t^4}{1-t}\oplus t^5\obp
$$
and its semigroup is $\sss=\{0,2,4,\to\}$. Set $\fff=\mathcal{O}_C\langle1,z\rangle$ where $z=t^2/(1-t)$. We have that $\deg_P(\fff)=0$ since $z\in \op$; while, following the proof of the above lemma, $\deg_U(\fff)=1+(2-1)=2$, so $\deg(\fff)=2$.

Now set $\mathcal{G}:=\oo\langle 1,t^2\rangle$. Note that we may write $z=t^2/(1-t)=t^2+t^3+\ldots$ by taking its power series. Now both $t^2$ and $z$ are in $\mathcal{G}_P$. Thus, as $t^2-z=t^3+\ldots$, we have that $v(\mathcal{G}) \setminus \sss = \{ 3 \}$, so $\deg_{P}(\mathcal{G})=1$. On the other hand, $\deg_{U}(\mathcal{G})=2$ and hence $\deg(\mathcal{G})=3$. Therefore, $\deg{\fff}<\deg{\mathcal{G}}$ and this stands for a counterexample for the proof of \eqref{gonality}. Let us proof that it is also a counterexample for the statement of \eqref{gonality}. Indeed, by \cite{KM2}*{Prop.~3.2.(1)} we have that $\gon(C)=1$ iff $g=0$. But here $g=2$, so $\gon(C)\geq 2$. Now, $\deg(\fff)=2$, so $\gon(C)=2$. So it suffices to show, for all $r \in \mathbb{N}$, that $\deg(\mathcal{H}_{r}) \geq 3$, where $\mathcal{H}_r=\oo \langle 1, t^{r} \rangle$. If $r=1$, we have that $\deg_{P}(\mathcal{H}_{1})= \# (v(\mathcal{H}_{1}) \setminus \sss)=2$; while $\deg_{Q}(\mathcal{H}_{1})=1$, so $\deg(\mathcal{H}_{1}) \geq 3$. If $r=2$, $\deg(\mathcal{H}_{2})=3$, as seen above. And if $r \geq 3$, note that $\deg(\mathcal{H}_{r}) \geq \deg_{Q}(\mathcal{H}_{r})=r \geq 3$.
\end{exam}

Based on the prior result, we now characterize trigonal monomial curves.

\begin{thm} \label{trigonalchar}
Let $C$ be a unicuspidal monomial curve with semigroup $\sss$. Then $C$ is trigonal if and only if, either
\begin{itemize}
\item[(i)] $\sss = \{ 0, \alpha, \alpha+1, \cdots, \alpha+k, \alpha+k+\ell,\to \}$ for $\alpha\geq 3$, $k\geq 0$, and $\ell\geq 2$;
%If so, $C$ is nearly Gorenstein if and only if $C'\ncong\mathbb{P}^1$; 
\item[(ii)] $\sss = \{0, \alpha, \alpha+2, \cdots, \alpha+2k,\to\}$ for $\alpha\geq 3$, and $k\geq 1$;
%If so, $C$ is nearly Gorenstein;
\item[(iii)] $\alpha=3$ and $\alpha\neq\beta$. 
% If so, $C$ is nearly Gorenstein if and only if $C$ is Kunz. 
\end{itemize}
\end{thm}
\begin{proof}
By \eqref{gonality}, the gonality of $C$ can be computed by a sheaf $\fff:=\mathcal{O}_{C} \langle1, t^{r} \rangle$ with $r>0$. For all $R \in C$, we have that
$$
\fff_{R}=\mathcal{O}_{R}+t^{r}\mathcal{O}_{R}.
$$
If $R\in C\setminus\{P,Q\}$, then $\fff_R=\oo_R$, and thus $\deg_R(\fff)=0$. On the other hand, $\fff_{Q}=t^{r}\mathcal{O}_{Q}$, thus $\deg_Q(\fff)=r$. Also, $\deg_{P}(\fff)=\# D$ where $D:=v(\fff_{P}) \setminus \sss$. So
\begin{equation} \label{degfff}
    \deg(\fff)= r + \#D.
\end{equation}
Now, as $C$ is monomial, $v(\fff_{P})=\sss+r$, so $D=(\sss+r)\setminus \sss$. 

Note that if $\alpha=1$ then $\gon(C)=1$; if $\alpha=2$, then $\gon(C)=2$ computed by $\oo_{C}\langle 1, t^2 \rangle$; and if $\alpha=\beta$ then $\gon(C)=2$ computed by $\oo_C\langle 1,t\rangle$. Thus we can further assume that $\alpha\geq 3$; $\alpha\neq \beta$; and, by \eqref{degfff}, $r\leq 3$ if $\fff$ computes gonality. 

\noindent {\bf case (i):} $r=1$.

\noindent Assume $\fff=\mathcal{O}_{C}\langle 1, t \rangle$ computes trigonality. Note that
\begin{align*}
    \deg(\fff)=3 &\Longleftrightarrow \#D=2 \\
    & \Longleftrightarrow \sss=\{ 0, \alpha, \alpha+1, \cdots, \alpha+k, \alpha+k+\ell,\to \}.
\end{align*}
for $k\geq 0$ and $\ell\geq 2$. Thus $\sss$ agrees with (i) in the statement of the theorem. Conversely, let us show that a curve with such an $\sss$ cannot have lower gonality. First note that $\deg(\oo_{C}\langle 1, t^r \rangle) \geq 3$ for $r\geq 3$. So we only need to test $r=2$. But $\deg (\oo_{C}\langle 1, t^2 \rangle)=2$ if and only if $\sss=\{0,2,4,\to \}$, which is precluded as $\alpha\geq 3$.

% Now note that $\kk=\{0, \ldots, \gamma-(\alpha+k+1) \} \cup \{\gamma-\alpha+1, \ldots, \gamma-1\} \cup \{\gamma+1, \to \}$. Then, by \eqref{ngumc}, it is easily seen that $C$ is nearly Gorenstein if and only if $\gamma-(\alpha+k+1)=0$. Indeed, otherwise $1\in\kk$, and hence $\gamma-(\alpha+k)\in\langle \kk\rangle$. But $\gamma-(\alpha+k)\not\in\kk\cup\{\gamma\}$. Now $\gamma-(\alpha+k+1)=0$ iff $\alpha+k=\gamma-1$ iff $\gamma-1\in\sss$ iff $1\not\in\kk$ iff $C'\ncong \mathbb{P}^1$. 

\noindent {\bf case (ii):} $r=2$.

\noindent Assume $\fff=\mathcal{O}_{C}\langle 1, t^{2} \rangle$ computes trigonality. Note that
\begin{align*}
    \deg(\fff)=3 &\Longleftrightarrow \#D=1 \\
    & \Longleftrightarrow \sss = \{0, \alpha, \alpha+2, \cdots, \alpha+2k,\to\}
\end{align*}
for $k\geq 1$. Thus $\sss$ agrees with (ii). To show a $C$ with such an $\sss$ is trigonal, we need to test only $r=1$. But it is easily seen that $\deg(\oo_{C}\langle 1, t \rangle) \geq 3$.

% Now note that $\kk=\{0,2,4,6,\ldots, \gamma-\alpha-1 \} \cup \{\gamma-\alpha+1, \ldots, \gamma-1\} \cup \{\gamma+1, \to \}$. Then, clearly, $\langle\kk\rangle=\kk\cup\{\gamma\}$ and hence $C$ is nearly Gorenstein by \eqref{ngumc}. 

\noindent {\bf case (iii):} $r=3$.

\noindent Assume $\fff=\mathcal{O}_{C}\langle 1, t^{3} \rangle$ computes trigonality. Note that
\begin{align*}
    \deg(\fff)=3 &\Longleftrightarrow \#D=0 \Longleftrightarrow \alpha=3 
\end{align*}
Thus $\sss$ agrees with (iii). To get the converse, just note that one can easily check that $\deg(\oo_{C}\langle 1, t^r \rangle) \geq 3$ if $r\geq 1,2$ for any $C$ with such an $\sss$. \end{proof}

% If $C$ is nearly Gorenstein, it is not Gorenstein, so there exists $a\leq b \in \gap$ such that $a+b=\gamma$.  Say $\gamma\equiv i\ {\rm mod}\ 3$ and $a\equiv j\ {\rm mod}\ 3$. Clearly $i,j\in\{1,2\}$ since $\gamma$ and $a$ are not in $\sss$. Also, $i\neq j$ since $\gamma-a\not\in\sss$. Therefore $2a\equiv\gamma\ {\rm mod}\ 3$ and hence $\gamma-2a\in\sss$, which implies $2a\not\in\kk$. By \eqref{ngumc}, we have $2a=\gamma$, i.e., $a=b=\gamma/2$. As $a,b$ are arbitrary, it follows that $C$ is Kunz. The converse is a general fact. 

Now we get back to Clifford dimension, studying plane monomial curves. We further show that those curves are
not the only ones of Clifford dimension 2.

\begin{thm} \label{cliffgthm}
The following hold.
\begin{itemize}
\item[(i)] If $C$ is a plane unicuspidal monomial curve, then
\begin{itemize}
    \item[(a)] $\cliffd(C)=2$, $\cliff(C)=d-4$, and $\cliff(C)$ is computed by an invertible sheaf;
    \item[(b)] $\gon(C)=d-1$ and is computed by both a base point free pencil and by a pencil with an irremovable base point. 
\end{itemize}
\item[(ii)] For $\alpha\geq 4$, let $C:=\left(1:t^{\alpha}:t^{2\alpha-1}:t^{3\alpha-2}: \cdots :t^{\alpha^{2}-(\alpha-1)}:t^{\alpha^{2}-(\alpha-1)+1}\right)$. Then: (a) $C$ is nonplanar; (b) ${\rm mult}_P(C)=\alpha$; (c) $\cliffd(C)=2$; (d) $\cliff(C)$ is computed only by non-invertible sheaves, and (e) $\cliff(C)=\alpha-3$. 
\end{itemize}
\end{thm}

\begin{proof}
{\bf To prove (i).(a)},
If $C$ is plane and monomial, then $C=(1:t^{\alpha}:t^{b})\subset\mathbb{P}^2$, where $\alpha$ is the multiplicity of $P=(1:0:0)$. But as $C$ is unicuspidal, i.e., $P$ is its unique singular point, then $b=\alpha+1$. Thus the semigroup of $C$ is $\sss=\langle \alpha,\alpha+1\rangle$. So let us first describe the structure of $\sss$. To this end, for short, for any integers $a,b$, we set $[a,b]:=[a,b]\cap\nn$. So we may write
\begin{equation} \label{structureofS}
   [1,\gamma]=\gap_1\cup \sss_1\cup\gap_2\cup\ldots\cup\sss_{\alpha-2}\cup\gap_{\alpha-1} 
\end{equation}
with $\gap_{i}=[(i-1)\alpha+i\, ,\, (i-1)\alpha+(\alpha-1)] \subset \gap$ and $\sss_i=[i\alpha\, ,\, i\alpha+i]\subset\sss$. Note, in particular, that $\#(\sss_i)=i+1$ and $\#(\gap_{i})=\alpha-i$. 
%Suponha que você tem n lacunas livres em um determinado passo. Se no passo anterior, você tiver mais que n+1 livres, no passo anterior do anterior, você já deve ter n+3 livres. 

%Also, $C$ is plane of degree $d=\alpha+1$, so its genus is
%$$
%g=\frac{(d-1)(d-2)}{2}=\frac{\alpha(\alpha-1)}{2}
%$$
%which can be computed by the sum $\sum_{i=1}^{\alpha-1}=\#(\gap_i)$ as well.

%Let $\fff:=\oo_{C}\langle 1, t^{\alpha}, t^{\alpha+1} \rangle$. Plainly, $\fff$ is invertible as $\fff_P=\oo_P$. So, by \eqref{cliffg2}, $\cliff(\fff)=\#(\gap_1)-\#(\sss_2)=(\alpha-1)-2=\alpha-3$. Also $H^0(\fff)=\langle 1, t^{\alpha},t^{\alpha+1}\rangle$, thus $h^0(\fff)=3$. So, by Riemmann-Roch, 
%\begin{align*}
 %   h^1(\fff) &= h^0(\fff)-\deg(\fff)-1+g\\
 %   &=3-(\alpha+1)-1+\bigg(\frac{\alpha(\alpha-1)}{2}\bigg) \\
  %  &=\frac{(\alpha-1)(\alpha-2)}{2}\geq 2
%\end{align*}
%since $\alpha\geq 4$.
%So $\fff$ contributes to the Clifford index.

Let $\fff:=\oo_{C}\langle 1, t^{a_{1}}, \ldots, t^{a_{n}} \rangle$ be a sheaf on $C$ that contributes to the Clifford index. Set $\ff :=v(\fff_P)$. Let $j$ be the largest integer such that $\gap_j\setminus\ff\neq\emptyset$, and say $e:=\#(\gap_j\setminus\ff)$. Then note that $\#(\gap_{j-i}\setminus\ff)\geq i+e$ for $0\leq i\leq j-1$. Indeed, this is a consequence of the following fact: if $\ell \in G_{i}\setminus\ff$ then, clearly, $\ell-\alpha, \ \ell-\alpha-1 \in \gap_{i-1}\setminus \ff$ because $\sss+\ff\subset\ff$ since $\fff_P$ is an $\op$-module. Set $k:=\#(\ff\cap \gap_1)$. So, in particular, taking $i=j-1$, we have $j-1+e\leq \#(\gap_1\setminus\ff)=\alpha-1-k$. Thus $j\leq\alpha-k-e$.

We claim that $\gap_{\alpha-i}\subset\ff$ for $1 \leq i \leq k$. Indeed, suppose $\gap_{\alpha-m}\setminus\ff\neq\emptyset$ for some $1\leq m\leq k$. Then $\alpha-m\leq j$, so $
\#(\gap_{(\alpha-m)-i}\setminus\ff)=j-(\alpha-m)+i+e\geq i+1.
$
Taking $i=\alpha-m-1$, we have that $\#(\gap_{1}\setminus\ff)\geq \alpha-m-1+1=\alpha-m\geq \alpha-k$. But $\#(\ff\cap \gap_1)=\alpha-1-k$, a contradiction. So the claim follows.

Assume first that $a_{n} \in \sss_{l}$ for some $l$. As $\fff$ contributes to the Clifford index, \eqref{equh1f} yields $j\geq l+1$. Set ${\rm E}:=\bigcup_{i=1}^n(a_i+\sss)$. Then $\ff=\sss\cup{\rm E}$. 
%. Finally, from the prior paragraph, we have that $j\leq \alpha-k-1$. If equality holds, the reader can check that we may just disregard the term below where the sum ranges from $j+1$ to $\alpha-k-1$. 
%$\ff\setminus\sss={\rm E}\setminus\sss$ and $G_i\cap\ff=G_i\cap{\rm E}$ for all $1\leq i\leq\alpha-1$. 
By \eqref{cliffg}, we get
{\allowdisplaybreaks
\begin{align*}
\cliff(\fff)&=\#((\gap\setminus{\rm E})\cap[1,a_n])-\#(\sss\cap[1,a_n])+\#(({\rm E}\setminus\sss)\cap[a_n+1,\gamma]) \\
&=\sum_{i=1}^{l}(\#(\gap_{i}\setminus{\rm E})-\#\sss_{i})+(l\alpha+l-a_{n})+\sum_{i=l+1}^{j} \#({\rm E}\cap \gap_{i})+\sum_{i=j+1}^{\alpha-1} (\alpha-i) \\
&= ((\alpha-k-1)-2)+\sum_{i=2}^{l}(\#(\gap_{i}\setminus{\rm E})-\#\sss_{i})+(l\alpha+l-a_{n}) \\
&\ \ \ \ \ \ \ \ \ \ \ \ \ \ \ \ \ \ \ \ \ \ + \sum_{i=l+1}^{j} \#({\rm E} \cap \gap_{i})+\sum_{i=j+1}^{\alpha-1} (\alpha-i) \\
& \geq (\alpha-3-k)+\sum_{i=2}^{l}((j-i+e)-(i+1))+(l\alpha+l-a_{n}) \\
&\ \ \ \ \ \ \ \ \ \ \ \ \ \ \ \ \ \ \ \ \ \
+ \sum_{i=l+1}^{j} \#({\rm E} \cap \gap_{i})+\sum_{i=j+1}^{\alpha-k-1} (\alpha-i)+\sum_{i=\alpha-k}^{\alpha-1}(\alpha-i) \\
&= (\alpha-3-k)+\sum_{i=2}^{l}(j+e-1-2i)+(l\alpha+l-a_{n}) \\
&\ \ \ \ \ \ \ \ \ \ \ \ \ \ \ \ \ \ \ \ \ \ + \sum_{i=l+1}^{j} \#({\rm E} \cap \gap_{i})+\sum_{i=j+1}^{\alpha-k-1} (\alpha-i)+\sum_{i=\alpha-k}^{\alpha-1}(\alpha-i) \\
&=(\alpha-3-k)+\big((l-1)(j+e-1)-(l-1)(l+2)\big)+(l\alpha+l-a_{n}) \\
&\ \ \ \ \ \ \ \ \ \ \ \ \ \ \ \ \ \ \ \ \ \ + \sum_{i=l+1}^{j} \#({\rm E} \cap \gap_{i})+\sum_{i=j+1}^{\alpha-k-1} (\alpha-i)+\sum_{i=\alpha-k}^{\alpha-1}(\alpha-i) \\
&=(\alpha-3)+\bigg((l-1)\big((j+e-1)-(l+2)\big)\bigg)+\big(l\alpha+l-a_{n}\big) \\
&\ \ \ \ \ \ \ \ \ \ \ \ \ \ + \sum_{i=l+1}^{j} \#({\rm E} \cap \gap_{i})+\sum_{i=j+1}^{\alpha-k-1} (\alpha-i)+\left(\left(\sum_{i=\alpha-k}^{\alpha-1}(\alpha-i)\right)-k\right) \\
& \geq \alpha-3.
\end{align*}
}
Indeed, the third to sixth terms in the two lines before last above are all trivially non-negative. Let us show the second term is non-negative as well. Note that it suffices to prove that $(j+e-1)-(l+2)\geq 0$, or equivalently, $(j-(l+1))+(e-2)\geq 0$. Now, by \eqref{equh1f}, $h^{1}(\fff)=\#(\gap')$ and $h^{1}(\fff) \geq 2$ since $\fff$ contributes to the Clifford index. But $\gap'=(\gap\setminus {\rm E})\cap[a_n+1,\gamma]$. So write $\#\gap'=\sum_{i=l+1}^j\#(\gap_i\setminus{\rm E})\geq 2$. Also, $1\leq e=\#(G_j\setminus{\rm E})$. So if $j=l+1$ then $e\geq 2$ and we are done. Otherwise, $j\geq l+2$ and we are done again since $e\geq 1$.

Thus we have proved that $\cliff(\fff)\geq \alpha-3$. Assume equality holds. Hence $\sum_{i=l+1}^{j} \#({\rm E} \cap \gap_{i})=0$, and then $k=0$, because if $k\neq 0$ then $\#(E\cap\gap_i)\geq 1$ for every $2\leq i\leq \alpha-1$. This kills the sixth term. Also, $\sum_{i=j+1}^{\alpha-1} (\alpha-i)$ cannot appear. But this holds if and only if $j=\alpha-1$, which is equivalent to saying $\ff=\sss$. But if so, then $\fff_P=\op$, that is, $\fff$ is invertible. Finally, if equality holds, then $a_{n}=l\alpha+l$ and, either $l=1$, or $(j+e-1)-(l+2)=0$. But $j=\alpha-1$, and so $e=1$. Summing up, we have $a_{n}=l\alpha+l$ with $l=1$ or $l=\alpha-3$.

Thus, among all sheaves of the form 
\begin{equation}
\label{equmos}
 \fff:=\oo_{C}\langle 1, t^{a_{1}}, \ldots, t^{a_{n}} \rangle   
\end{equation} 
contributing to Clifford index, with $a_n\in\sss$, we found only two of minimal Clifford index. Both are invertible, so they can be written as $\fff=\oo\langle 1, t^{a_n}\rangle$. Those are  $\fff_1=\oo_C\langle 1,t^{\alpha+1}\rangle$ and $\fff_2=\oo_C\langle 1,t^{(\alpha-3)(\alpha+1)}\rangle$. But one can check by \eqref{equh0f} that 
$h^0(\fff_1)=3$ and $h^0(\fff_2)=((\alpha-1)(\alpha-4)/2)+(\alpha-1)\geq 3$. 

Now assume $a_{n} \in \gap_{l}$ and set $j,e, k$ as above. Then, similarly, \eqref{cliffg} yields
{\allowdisplaybreaks
\begin{align*}
\cliff(\fff) &=\sum_{i=1}^{l-1}(\#(\gap_{i}\setminus{\rm E})-\#\sss_{i})+\#((\gap_l\setminus {\rm E})\cap[(l-1)\alpha+l,a_{n}]) \\
& \ \ \ +\#(({\rm E}\cap\gap_{l})\cap[a_{n}+1,(l-1)\alpha+\alpha-1])+
\sum_{i=l+1}^{j} \#({\rm E} \cap \gap_{i})+\sum_{i=j+1}^{\alpha-1} (\alpha-i) \\
& \geq (\alpha-3)+\bigg((l-2)((j+e-1)-(l+1))\bigg)+\#((\gap_l\setminus {\rm E})\cap[(l-1)\alpha+l,a_{n}]) \\
& \ \ \ +\#(({\rm E}\cap\gap_{l})\cap[a_{n}+1,(l-1)\alpha+\alpha-1])+ \sum_{i=l+1}^{j} \#({\rm E} \cap \gap_{i}) \\
&\ \ \ \ \ \ \ \ \ \ \ \ \ \ \ \ \ \ \ \ \ + \sum_{i=j+1}^{\alpha-k-1} (\alpha-i)+\left(\left(\sum_{i=\alpha-k}^{\alpha-1}(\alpha-i)\right)-k\right) 
\end{align*}
}
Call ($\ast$) the last three lines above. Again, note that all third to seventh terms in ($\ast$) are trivially non-negative. Assume $l\geq 2$. Then the second term is non-negative too because $(j+e-1)-(l+1)\geq 0$ since $j\geq l+1$ and $e\geq 1$. But, clearly, $\#({\rm E}\cap\gap_i)\geq 1$ for every $l\leq i\leq \alpha-1$. Thus ($\ast$) is strictly greater than $\alpha-3$. Assume $l=1$, then 
%{\allowdisplaybreaks
%\begin{align*}
%&  (\alpha-3)+(3-j-e)+(a_n-k)+\sum_{i=2}^{j} \#({\rm E} \cap \gap_{i})+\sum_{i=j+1}^{\alpha-k-1} (\alpha-i)+\sum_{i=\alpha-k}^{\alpha-1}(\alpha-i)-k\\
%&\geq (\alpha-3)+(3-j-e)+(j-1)+(\alpha-j-1)= (\alpha-3)+1+((\alpha-j)-e)>\alpha-3
%\end{align*}
%}
%Acrescentei(1) 
{\allowdisplaybreaks
\begin{align*}
   \cliff(\fff)= \#((\gap_{1}\setminus{\rm E}) \cap [1,a_{n}])+\#(({\rm E}\setminus \sss)\cap[a_{n+1},\gamma])
\end{align*}
}
But note that as $a_{n} \in \gap_{1}$, then ${\rm E} \cap \gap_{i} \neq \emptyset$ for $2 \leq i \leq \alpha-1$. Therefore, the second term is at least $\alpha-2$. Hence $\cliff(\fff) > \alpha-3$ if $l=1$.

Thus, all sheaves of the form \eqref{equmos} contributing to Clifford index, with $a_n\not\in\sss$, have Clifford index strictly greater than $\alpha-3$. Now by \eqref{gonality}, the Clifford dimension of $C$ can be computed by a sheaf like \eqref{equmos}. So $\cliffd(C)=2$ and $\cliff(C)=\alpha-3$. But, plainly, $C$ has degree $d=\alpha+1$, which implies $\cliff(C)=d-4$. 

\textbf{To prove (i)(b)}, by \eqref{gonality}(ii), the gonality of $C$ can be computed by a sheaf of the form $\fff = \oo\langle 1, t^{a} \rangle$ for some $a \in \nn^{\ast}$. Set ${\rm E}=a+\sss$. Then
$$
\deg(\fff)=a+\#({\rm E}\setminus\sss)
$$
So the smallest degree of $\fff$ for $a\in\sss$ occurs when $a=\alpha$, which yields $\deg(\fff)=\alpha$. Thus, if $a\not\in\sss$, we may assume $1\leq a\leq \alpha-1$. But if so, $\# (\gap_{i} \cap {\rm E}) \geq 1$ for all $1 \leq i \leq \alpha-1$, and is exactly $1$ if $a=1$. Hence, the smallest degree of $\fff$ for $a \notin \sss$ occurs when $a=1$, which yields $\deg(\fff)=1+(\alpha-1)=\alpha$. Now $|\oo \langle 1, t^{\alpha} \rangle|$ is base-point free since $\oo \langle 1, t^{\alpha} \rangle$ is invertible, while $|\oo \langle 1, t \rangle|$ has an irremovable base point at $P$ since $\oo \langle 1, t\rangle_P$ is not free.

%\begin{align*}
%    \deg \oo\langle 1, t^{a} \rangle & = \deg_{P} \oo\langle 1, t^{a} \rangle + \deg_{\infty} \oo\langle 1, t^{a} \rangle \\
%    & = \left(\sum_{i=1}^{a}2i+a(\alpha-1-2a) \right) + a \\
%    & = (a+1)(a)+a(\alpha-1-2a) + a \\
%    & = a(\alpha-a+1).
%\end{align*}
%The second line is obvious for $a \leq  \left \lfloor \dfrac{\alpha}{2} \right \rfloor$. To see that it still holds otherwise, assume with no loss of generality that $\alpha$ is even. Then, 
%\begin{align*}
%    \deg \oo\langle 1, t^{a} \rangle & = \deg_{P} \oo\langle 1, t^{a} \rangle + \deg_{\infty} \oo\langle 1, t^{a} \rangle \\
%    & = \left(\sum_{i=1}^{\left \lfloor (\alpha)/2 \right \rfloor}2i+\sum_{i=\left \lfloor (\alpha)/2 \right \rfloor+1}^{a}2i+a(\alpha-1-2a) \right) + a \\
%     & = \left(\sum_{i=1}^{ \alpha/2 }2i+\sum_{i=(\alpha)/2 +1}^{a}2i+a(\alpha-1-2a) \right) + a \\
%    & = \left(2\left(\dfrac{\alpha}{4}\right)\left(\dfrac{\alpha}{2}+1\right)+\left(a-\dfrac{\alpha}{2}\right)\left(2\left(\dfrac{\alpha}{2}+1\right)+\left(a-\dfrac{\alpha}{2}-1\right)\right)+a\left(\alpha-1-2a\right)\right)+a \\
%    & = a(\alpha-a+1).
%\end{align*}

% In particular, note that $\deg \oo\langle 1, t \rangle $ = $\deg_{P} \oo\langle 1, t \rangle +\deg_{\infty} \oo\langle 1, t \rangle = (\alpha-1)+1=\alpha$ and $\deg \oo\langle 1, t^{\alpha} \rangle $ = $\deg_{\infty} \oo\langle 1, t^{\alpha} \rangle = \alpha$. For a fixed $\alpha$, these are the minimums of the associated function in the interval $[1, \alpha]$. Hence, $\gon(C)=\alpha=d-1$. 

{\bf To prove (ii)}, consider the following family of monomial unicuspidal curves $C:=(1:t^{\alpha}:t^{2\alpha-1}:t^{3\alpha-2}: \cdots :t^{\alpha^{2}-(\alpha-1)}:t^{\alpha^{2}-(\alpha-1)+1}) \subset \mathbb{P}^{\alpha+1}$ with semigroup $\sss=\langle \alpha, 2\alpha-1, 3\alpha-2, \ldots, \alpha^{2}-(\alpha-1) \rangle$. Note that the last component $t^{\alpha^{2}-(\alpha-1)+1}$ guarantees that $C$ is unicuspidal. Clearly, $C$ is non-planar (in fact, $C$ is non-Gorenstein) and ${\rm mult}_{P}(C)=\alpha$. Note that $\sss$ is such that $m\alpha-1+(\sss\setminus\{0\})\subset\sss$ for $m\in\nn^{\ast}$. Let $a$ be the greatest integer such that  $a_n>a\alpha-1$, then, clearly, $\cliff(\fff+\sum_{i=1}^a t^{i(\alpha-1)}\oo) \leq \cliff(\fff)$. So we may assume $t^{i(\alpha-1)}$ is a global section of $\fff$ for all $1\leq i\leq a$. In particular, $\cliff(C)$ is computed only by non-invertible sheaves. 

As in \eqref{structureofS}, write $\sss$ as $\sss_{i}=[i\alpha-(i-1), i\alpha]$ and $\gap_{i}=[(i-1)\alpha+1,i(\alpha-1)]$. Also, $\#(\sss_i)=i$ and $\#(\gap_{i})=\alpha-i$. \par
Let $\fff:=\oo_{C}\langle 1, t^{a_{1}}, \ldots, t^{a_{n}} \rangle$ be a sheaf that contributes to the Clifford index. Set $\ff:=v(\fff)$. Assume the largest integer outside $\ff$ is in $\gap_j$, and say $\#(\gap_j\setminus\ff)=e$. Note that $\#(\gap_{j-i}\setminus\ff)\geq i+e-1$ for $1\leq i\leq j-1$. \par
Note also that if $\#\left((\ff\cap \gap_1) \cap [1,\alpha-2]\right)=k$, then $\gap_{\alpha-i}\subset\ff$ for $1 \leq i \leq k$. Considering this adjustment, the same observations of the previous item hold. So, assume first that $a_{n} \in \sss_{l}$ for some $l$. So \eqref{cliffg} yields
{\allowdisplaybreaks
\begin{align*}
\cliff(\fff)&=\#((\gap\setminus{\rm E})\cap[1,a_n])-\#(\sss\cap[1,a_n])+\#(({\rm E}\setminus\sss)\cap[a_n+1,\gamma]) \\
&=\sum_{i=1}^{l}(\#(\gap_{i}\setminus\ff)-\#\sss_{i})+(l\alpha-a_{n})+\sum_{i=l+1}^{j} \#(\ff \cap \gap_{i})+\sum_{i=j+1}^{\alpha-1} (\alpha-i) \\
&= ((\alpha-1)-(k+1))-1)+\sum_{i=2}^{l}(\#(\gap_{i}\setminus\ff)-\#\sss_{i})+(l\alpha-a_{n}) \\
&\ \ \ \ \ \ \ \ \ \ \ \ \ \ \ \ \ \ \ \ \ \ + \sum_{i=l+1}^{j} \#(\ff \cap \gap_{i})+\sum_{i=j+1}^{\alpha-1} (\alpha-i) \\
& \geq (\alpha-3-k)+\sum_{i=2}^{l}((j-i+e-1)-i)+(l\alpha-a_{n}) \\
&\ \ \ \ \ \ \ \ \ \ \ \ \ \ \ \ \ \ \ \ \ \
+ \sum_{i=l+1}^{j} \#(\ff \cap \gap_{i})+\sum_{i=j+1}^{\alpha-k-1} (\alpha-i)+\sum_{i=\alpha-k}^{\alpha-1}(\alpha-i) \\
&= (\alpha-3-k)+\sum_{i=2}^{l}(j+e-2i-1)+(l\alpha-a_{n}) \\
&\ \ \ \ \ \ \ \ \ \ \ \ \ \ \ \ \ \ \ \ \ \ + \sum_{i=l+1}^{j} \#(\ff \cap \gap_{i})+\sum_{i=j+1}^{\alpha-k-1} (\alpha-i)+\sum_{i=\alpha-k}^{\alpha-1}(\alpha-i) \\
&=(\alpha-3-k)+\big((l-1)(j+e-1)-(l-1)(l+2)\big)+(l\alpha-a_{n}) \\
&\ \ \ \ \ \ \ \ \ \ \ \ \ \ \ \ \ \ \ \ \ \ + \sum_{i=l+1}^{j} \#(\ff \cap \gap_{i})+\sum_{i=j+1}^{\alpha-k-1} (\alpha-i)+\sum_{i=\alpha-k}^{\alpha-1}(\alpha-i) \\
&=(\alpha-3)+\bigg((l-1)((j+e-1)-(l+2))\bigg)+\big(l\alpha-a_{n}\big) \\
&\ \ \ \ \ \ \ \ \ \ \ \ \ \ \ \ \ \ \ \ \ \ + \sum_{i=l+1}^{j} \#(\ff \cap \gap_{i})+\sum_{i=j+1}^{\alpha-k-1} (\alpha-i)+\left(\left(\sum_{i=\alpha-k}^{\alpha-1}(\alpha-i)\right)-k\right) \\
& \geq \alpha-3.
\end{align*}}
Indeed, again, the second to sixth terms in the two lines before last above are all non-negative. For the second term, proceed exactly as in (i).(a).
%$j \geq l+1$ and $e\geq 1$, then $(j+e-1)-(l+2)\geq -1$ and the equality holds if and only if $e=1$ and $j=l+1$. If that happens, we have to make up for the $-1$ by playing with the other terms. Note that $j \leq \alpha - k +1-e$. So, if $e=1$, then $j \leq \alpha - k$. Suppose the fifth term vanishes. As $e=1$ and $\#(\gap_j\cap\ff)=\alpha-j-e$, this happens if and only if either $j=\alpha-1$ or $j \geq \alpha-k$. The sixth term vanishes if and only if $\alpha-1=\alpha-k$. So we conclude that $\alpha-2 \leq j \leq \alpha-1$. But $j=l+1$, so both cases imply $h^{1}(\fff)=1$ and $\fff$ does not contribute to the Clifford index, a contradiction. 

Now assume equality holds. Thus, $\sum_{i=l+1}^{j} \#(\ff \cap \gap_{i})=0$, then $k=0$. Also, $\sum_{i=j+1}^{\alpha-1} (\alpha-i)$ cannot appear. Therefore, $j=\alpha-1$. Finally, if equality holds, then $a_{n}=l\alpha+l$ and, either $l=1$, or $(j+e-1)-(l+2)=0$. But $j=\alpha-1$, and so $e=1$. Thus $a_{n}=l\alpha$ with $l=1$ or $l=\alpha-3$.

Thus, there are only two sheaves as in \eqref{equmos} that contribute to the Clifford index, with $a_n\in\sss$, of minimal Clifford index. Those are  $\fff_1=\oo_C\langle 1,t^{\alpha-1},t^{\alpha}\rangle$ and $\fff_2=\oo_C\langle 1,t^{\alpha-1},t^{2\alpha-2}, \ldots, t^{(\alpha-3)(\alpha-1)}, t^{(\alpha-3)\alpha}\rangle$. But $h^0(\fff_1)=3$ and $h^0(\fff_2)\geq 3$. 
%parei aqui

Now, assume $a_{n} \in \gap_{l}$ for some $l$ and set $k, j$ as above. Then \eqref{cliffg} yields
{\allowdisplaybreaks
\begin{align*}
\cliff(\fff) &=\sum_{i=1}^{l-1}(\#(\gap_{i}\setminus{\rm E})-\#\sss_{i})+\#((\gap_l\setminus {\rm E})\cap[(l-1)\alpha+1,a_{n}]) \\
& \ \ \ +\#(({\rm E}\cap\gap_{l})\cap[a_{n}+1,l(\alpha-1)])+
\sum_{i=l+1}^{j} \#({\rm E} \cap \gap_{i})+\sum_{i=j+1}^{\alpha-1} (\alpha-i) \\
& \geq (\alpha-3)+\bigg((l-2)((j+e-1)-(l+1))\bigg)+\#((\gap_l\setminus {\rm E})\cap[(l-1)\alpha+1,a_{n}]) \\
& \ \ \ +\#(({\rm E}\cap\gap_{l})\cap[a_{n}+1,l(\alpha-1)])+ \sum_{i=l+1}^{j} \#({\rm E} \cap \gap_{i}) \\
&\ \ \ \ \ \ \ \ \ \ \ \ \ \ \ \ \ \ \ \ \ + \sum_{i=j+1}^{\alpha-k-1} (\alpha-i)+\left(\left(\sum_{i=\alpha-k}^{\alpha-1}(\alpha-i)\right)-k\right) \\
\end{align*}
}
Again, call ($\ast$) the last three lines above, and note that all third to seventh terms in ($\ast$) are trivially non-negative. Assume $l\geq 2$. Then the second term is non-negative too because $(j+e-1)-(l+1)\geq 0$. But $\#({\rm E}\cap\gap_i)\geq 1$ for every $l\leq i\leq \alpha-1$. Thus ($\ast$) is strictly greater than $\alpha-3$. Assume $l=1$, then
{\allowdisplaybreaks
\begin{align*}
   \cliff(\fff)= \#((\gap_{1}\setminus{\rm E}) \cap [1,a_{n}])+\#(({\rm E}\setminus \sss)\cap[a_{n+1},\gamma])
\end{align*}
}
But note that if $\gap_{1} \cap [1,\alpha-2] \neq \emptyset$, then ${\rm E} \cap \gap_{i}  \neq \emptyset$ for $2 \leq i \leq \alpha-2$. Therefore, the second term is at least $\alpha-2$. Hence $\cliff(\fff) > \alpha-3$. On the other hand, if $\gap_{1} \cap [1,\alpha-2] = \emptyset$, then $\fff = \oo\langle 1, t^{\alpha-1} \rangle $. In this case, $\cliff(\fff)=\alpha-2>\alpha-3$. 

Thus, all sheaves of the form \eqref{equmos} contributing to Clifford index, with $a_n\not\in\sss$, have Clifford index strictly greater than $\alpha-3$. Therefore the clifford index is computed with $a_{n} \in \sss$, which yields $\cliff(C)=\alpha-3$ and $\cliffd(C)=2$.  
\end{proof}

%Given a semigroup $\sss$, break down $\mathbb{N}\cap[0,\gamma]$ into ordered blocks $M_i$ of $m_i$ consecutive elements of $\sss$, and $L_i$ of $l_i$ consecutive gaps in ${\rm G}$. In other words, 
%$$
%\sss=\{0\}\cup \bigg(\bigcup_i M_i\bigg)\cup {\rm C} 
%$$
%where
%$$
%M_i=\left\{\sum_{k=1}^i l_k+\sum_{k=1}^{i-1}m_{k}+1, \ldots, \sum_{k=1}^i %l_k+\sum_{k=1}^{i} m_{k} \right\}.
%$$
\begin{exam}
    Now we give a simple example to illustrate \eqref{cliffgthm}(ii). Consider the curve $C:=(1: t^{5}:t^{9}:t^{13}:t^{17}:t^{21}:t^{22})$. It is unicuspidal due to the presence of $t^{22}$, which was inserted to avoid a singularity at $(0:\ldots:0:1)$. Its semigroup is $\sss = \langle 5, 9, 13, 17, 21 \rangle$ and the set of gaps $\gap$, as in the proof of \eqref{cliffgthm}.(ii), can be broken down into $\gap_1=\{1,2,3,4\}$, $\gap_2=\{6,7,8\}$, $\gap_3=\{11,12\}$ and $\gap_4=\{16\}$. So $C$ has genus $10$ and it is non-Gorenstein as the conductor is $17\neq 2g=20$. The gonality of $C$ is computed by the sheaves $\oo\langle 1,t\rangle$, $\oo\langle 1,t^4\rangle$ and $\oo\langle 1,t^5\rangle$, all of degree $5$, the first ones non-invertible and the last one locally free. As they compute gonality, they all have Clifford index $\gon(C)-2=3$, which can easily be checked directly. The Clifford index of $C$ is computed by $\fff:=\oo\langle 1,t^4,t^5\rangle$. We have $\deg(\fff)=6$ and $h^0(\fff)=3$. So $\cliff(C)=\cliff(\fff)=2$ and $\cliffd(C)=2$.
\end{exam}

\begin{exam}
\label{exaexc}
We finish this work with an example of a unicuspidal monomial curve with Clifford dimension $3$. Let $C=(1:t^6:t^8:t^9)$. It is unicuspidal with semigroup $\sss=\langle 6,8,9\rangle$. The set of gaps is $\gap=\{1,2,3,4,5,7,10,11,13,19\}$.  So $C$ has genus $g=10$ and it is Gorenstein as the conductor is $20=2g$. As in the prior example, the gonality of $C$ is computed by three sheaves generated by monomial sections, namely, $\oo\langle 1,t\rangle$, $\oo\langle 1,t^3\rangle$ and $\oo\langle 1,t^6\rangle$, all of degree $6$, and, again, the last one is the only invertible. 

Now we will show that $\cliffd(C)=3$. By \eqref{h1h1}, the Clifford dimension can be computed by a sheaf of the form $\fff = \mathcal{O}_{C} \langle 1, t^{a_{1}}, \ldots, t^{a_{n}} \rangle$. With no loss of generality, we may further assume $n=h^0(\fff)-1$. So $n=\cliffd(C)$. If $n=1$, then, by \eqref{ap3gonacliff}.(ii), we have $\cliff(C)=\gon(C)-2=4$, so $n\neq 1$.  Now note that for $\mathcal{G} := \oo_{C}\langle 1, t^{6}, t^{8}, t^{9} \rangle$, $\text{Cliff}(\mathcal{G}) = 9-2\times 3 = 3$, so $\cliff(C)\leq 3$. This implies that if $n=2$, then 
\begin{equation} \label{degreebound}
\deg (\fff) \leq 7\ \text{if}\ n=2, \text{and}\  \deg (\fff) \leq 2+2n \ \text{if}\ n\geq 4.  
\end{equation}
This easily shows that among all (eligible) invertible sheaves, $\mathcal{G}$ is of smallest Clifford index. Assume $\fff$ is non-invertible. Let $a$ be the smallest element in ${\rm A}:= \bigcup_{i=1}^n(a_i+\sss)\cap\gap$. By \eqref{h1h1}, we see that $a \in \{1,2,3,4,5,7,10\}$. Also, \eqref{h1h1} implies that if $a=1$, then $a_{n} \leq 4$; if $a=2,4,5$, then $a_{n} \leq 6$; if $a=3, 7$, $a_{n} \leq 9$; and if $a=10$, then $a_{n}=10$. One can check, after an easy though laborious case by case analysis, that those constraints on the $a_n$ for each $a$, along with the ones of \eqref{degreebound}, yield that $\mathcal{G}$ computes the Clifford dimension of $C$, which is $3$. Note that $C\subset\mathbb{P}^3=\{(w:x:y:z)\}$ is the intersection of the cubics $wz^2-x^3$ e $xz^2-y^3$. Also, the system $(\mathcal{G}, \langle 1, t^6,t^8,t^9\rangle)$ provides the Clifford embedding of $C$ as in \eqref{cliffgeral}(i). 
\end{exam}

\end{document}